\documentclass[a4paper,12pt,oneside]{article}
\usepackage[english]{babel}
\usepackage[T1]{fontenc} 
\usepackage[utf8]{inputenc}
\usepackage{amsthm}
\usepackage{amsmath}
\usepackage{amssymb}  
\usepackage{indentfirst}
\usepackage{fancyhdr}
\usepackage{amsthm}
\usepackage{graphicx}
\usepackage{pdfpages}
\usepackage{esint}
\usepackage{url}

\definecolor{dg}{rgb}{0.01, 0.75, 0.24}
\numberwithin{equation}{section}

\theoremstyle{plain} 
\newtheorem{thm}{Theorem}[section]

\newtheorem{lem}[thm]{Lemma} 
\newtheorem{prop}[thm]{Proposition} 
\newtheorem{rmk}[thm]{Remark}

\def\io{\int_{\Omega}}
\def\ibtt{\int_{B_{t'}}}
\def\ibt {\int_{B_{t}}}
\def\ibR {\int_{B_{R}}}
\def\R{\mathbb{R}}
\def\N{\mathbb{N}}
\def\A{\mathcal{A}}
\def\ib0 {\int_{B_{R/4}(0)}}
\def\ibk {\int_{E_k}}

\usepackage{geometry}
\allowdisplaybreaks[4]

\begin{document}

\title{ Higher fractional differentiability for solutions\\ to a class of obstacle problems \\with non-standard growth conditions}
\author {   
\sc{Antonio Giuseppe Grimaldi}\thanks{
 Dipartimento di Matematica e Applicazioni "R. Caccioppoli", Università degli Studi di Napoli "Federico II", Via Cintia, 80126 Napoli
 (Italy). E-mail: \textit{antoniogiuseppe.grimaldi@unina.it}}  \;\textrm{and}   Erica Ipocoana \thanks{University of Modena and Reggio Emilia, Dipartimento di Scienze Fisiche, Informatiche e Matematiche, via Campi 213/b, I-41125 Modena (Italy).
E-mail: \textit{erica.ipocoana@unipr.it}}}

\maketitle
\maketitle
\begin{abstract}
We here establish the higher fractional differentiability for solutions to a class of obstacle problems with non-standard growth conditions. We deal with
the case in which the solutions to the obstacle problems satisfy a variational inequality of the form
\begin{equation*}
\displaystyle\int_{\Omega} \langle \mathcal{A}(x,Du)  ,D(\varphi-u) \rangle dx \geq 0 \qquad \forall \varphi \in \mathcal{K}_\psi(\Omega),
\end{equation*}
where $\Omega$ is a bounded open subset of $\R^n$, $\psi \in W^{1,p}(\Omega)$ is a fixed function called \textit{obstacle} and $\mathcal{K}_{\psi}(\Omega)= \{ w \in W^{1,p}(\Omega) : w \geq \psi \ \text{a.e. in} \ \Omega  \}$ is the class of admissible functions. Assuming that the gradient of the obstacle belongs to some suitable Besov space, we are able to prove that some fractional differentiability property transfers to the gradient of the solution.
\end{abstract}


\section{Introduction}
The aim of this paper is the study of the higher fractional differentiability properties of the gradient of solutions $u \in W^{1,p}(\Omega)$ to obstacle problems of the form
\begin{gather}\label{!1}
\min \biggl\{ \displaystyle\int_{\Omega} F(x,Dw)dx \ : \ w \in \mathcal{K}_{\psi}(\Omega)  \biggr\},
\end{gather}
where $\Omega$ is a bounded open set of $\mathbb{R}^n$, $n \geq 2$.
\\The function $\psi: \Omega \rightarrow [-\infty, + \infty)$, called \textit{obstacle}, belongs to the Sobolev class $W^{1,p}(\Omega)$ and the class $\mathcal{K}_{\psi}(\Omega) $ is defined as follows
\begin{gather}\label{!2}
\mathcal{K}_{\psi}(\Omega)= \{ w \in W^{1,p}(\Omega) : w \geq \psi \ \text{a.e. in} \ \Omega  \}.
\end{gather}
Note that the set $\mathcal{K}_{\psi}(\Omega) $ is not empty since $\psi \in \mathcal{K}_{\psi}(\Omega) $.
\\In what follows, we assume that $F : \Omega \times \mathbb{R}^n \rightarrow [0, + \infty)$ is a Carath\'{e}odory function such that there exists a function $\tilde{F} : \Omega \times [0, + \infty) \rightarrow [0, + \infty)$ satisfying the following equality
$$ F(x, \xi)= \tilde{F}(x, |\xi|)\eqno{\rm{{ (F1)}}}$$
for a.e. $x \in \Omega$ and every $\xi \in \mathbb{R}^n$.
\\Moreover, we also assume that there exist positive constants $\tilde{\nu}$, $\tilde{L}$, $\tilde{l}$, exponents $2 \leq p < q < + \infty$ and a parameter $\mu \in [0,1]$, that will allow us to consider in our analysis both the degenerate and the non-degenerate situation, such that the following assumptions are satisfied:
$$ \dfrac{1}{\tilde{l}} (|\xi|^2-\mu^2 )^{\frac{p}{2}} \leq F(x, \xi) \leq \tilde{l}(\mu^2+|\xi|^2)^{\frac{q}{2}}\eqno{\rm{{ (F2)}}}$$
$$  \langle D_{\xi \xi}F(x, \xi) \lambda, \lambda \rangle \geq \tilde{\nu} (\mu^2+ |\xi|^2)^{\frac{p-2}{2}} |\lambda|^2  \eqno{\rm{{ (F3)}}}$$
$$ |D_{\xi \xi }F(x, \xi)| \leq \tilde{L}(\mu^2 + |\xi|^2)^{\frac{q-2}{2}}\eqno{\rm{{ (F4)}}}$$
for a.e. $x,y \in \Omega$ and every $\xi \in \mathbb{R}^n$.
\\Very recently, in \cite{cupini.marcellini.mascolo.passarelli} it has been proved that (F3) and (F4) imply (F2), i.e. if $p<q$, the functional $F$ has non-standard growth conditions of $p,q$-type, as initially defined and studied by Marcellini \cite{marcellini1986,marcellini1989,marcellini1991}.
In recent years there has been a considerable of interest in functionals with $p,q$-growth, see for instance \cite{beck.mingione, bildhauer.fuchs, byun.oh, byun.youn, defilippis.mingione}. Other results that deserve to be quoted are
\cite{carrozza.kristensen.passarelli,cupini.marcellini.mascolo,
eleuteri.marcellini.mascolo1,eleuteri.marcellini.mascolo2, 
eleuteri.marcellini.mascolo3}, for the case of elliptic equations, and 
\cite{bogelein.duzaar.marcellini1, bogelein.duzaar.marcellini2,
bogelein.duzaar.marcellini3,marcellini2019} for the case of parabolic equations. 
\\We remark that assumption (F1) is known in
the literature as Uhlenbeck structure and it was showed in \cite{uh} that it prevents the irregularity phenomenon in problems with non-standard growth.

We say that function $F$ satisfies assumption (F5) if there exist a non-negative function $k \in L^{r}_{\text{loc}}(\Omega)$, with $r > \frac{n}{\alpha}$ and $0 < \alpha <1$, such that
$$ |D_{\xi}F(x, \xi)- D_{\xi}F(y, \xi)| \leq |x-y|^{\alpha}(k(x)+k(y))(\mu^2+ |\xi|^2)^{\frac{q-1}{2}}\eqno{\rm{{ (F5)}}}$$
for a.e. $x,y \in \Omega$ and every $\xi \in \mathbb{R}^n$.
\\On the other hand, we say that assumption (F6) is satisfied if there exists a sequence of measurable non-negative functions $g_k \in L^{r}_{\text{loc}}(\Omega)$ such that
$$\displaystyle\sum_{k=1}^{\infty} \Vert g_k \Vert^{\sigma}_{L^{r}(\Omega)} < \infty,$$
and at the same time
$$|D_{\xi}F(x,\xi)-D_{\xi}F(y, \xi)| \leq |x-y|^{\alpha} (g_k(x)+g_k(y)) (\mu^2 +|\xi|^2)^{\frac{q-1}{2}} \eqno{\rm{{ (F6)}}} $$
for a.e. $x,y \in \Omega$ such that $2^{-k} \text{diam}(\Omega) \leq |x-y| < 2^{-k+1}\text{diam}(\Omega)$ and for every $\xi \in \mathbb{R}^n$.
\\It is worth observing that, in the case of standard growth conditions, i.e. $p=q$, $u \in W^{1,p}(\Omega)$ is a solution to the obstacle problem in $\mathcal{K}_{\psi}(\Omega) $ if, and only if, $u \in \mathcal{K}_{\psi}(\Omega) $ solves the variational inequality 
\begin{align}\label{1}
\displaystyle\int_{\Omega} \langle \mathcal{A}(x,Du)  ,D(\varphi-u) \rangle dx \geq 0 
\end{align}
for all $\varphi \in \mathcal{K}_{\psi}(\Omega)$, where we set
\begin{gather}
\mathcal{A}(x,\xi)= D_{\xi}F(x,\xi).
\end{gather}
This equivalence has been proved successfully in the case non-standard growth conditions by Eleuteri and Passarelli di Napoli in \cite{eleuteri.passarelli2021}.

From contiditions (F2)--(F4), we deduce the existence of positive constants $\nu,L,l$ such that the following $p$-ellipticity and $q$-growth conditions are satisfied by the map $\mathcal{A}$:
$$|\mathcal{A}(x,\xi)| \leq l(\mu^2 +|\xi|^2)^{\frac{q-1}{2}} \eqno{\rm{{ (A1)}}}$$
$$\langle \mathcal{A}(x,\xi)-\mathcal{A}(x,\eta), \xi-\eta \rangle \geq \nu |\xi-\eta|^{2} (\mu^2+|\xi|^{2}+|\eta|^{2})^{\frac{p-2}{2}} \eqno{\rm{{ (A2)}}}$$
$$|\mathcal{A}(x,\xi)-\mathcal{A}(x,\eta)|\leq L|\xi-\eta| (\mu^2+|\xi|^{2}+|\eta|^{2})^{\frac{q-2}{2}} \eqno{\rm{{ (A3)}}}$$
for a.e. $x,y \in \Omega$, for every $\xi, \eta \in \mathbb{R}^n$, where we recall that $0< \alpha< 1$. 
\\Furthermore, if condition (F5) or (F6) holds, then $\A$ satisfies assumptions (A4) or (A5), respectively, that is
$$|\mathcal{A}(x,\xi)-\mathcal{A}(y, \xi)| \leq |x-y|^{\alpha} (k(x)+k(y)) (\mu^2 +|\xi|^2)^{\frac{q-1}{2}} \eqno{\rm{{ (A4)}}}$$
for a.e. $x,y \in \Omega$ and every $\xi \in \mathbb{R}^n$, or
$$|\A(x,\xi)-\A(y, \xi)| \leq |x-y|^{\alpha} (g_k(x)+g_k(y)) (\mu^2 +|\xi|^2)^{\frac{q-1}{2}} \eqno{\rm{{ (A5)}}}$$
for a.e. $x,y \in \Omega$ such that $2^{-k} \text{diam}(\Omega) \leq |x-y| < 2^{-k+1}\text{diam}(\Omega)$ and for every $\xi \in \mathbb{R}^n$.

The obstacle problem appeared in the mathematical literature in the work of Stampacchia \cite{stampacchia} in the special case $\psi = \chi_E$ and related to the capacity of a subset $E \Subset \Omega$; in an earlier independent work, Fichera \cite{fichera} solved the first unilateral problem, the so-called \textit{Signorini problem} in elastostatics.

It is usually observed that the regularity of solutions to the obstacle problems is influenced by the one of the obstacle; for example, for linear obstacle problems, obstacle and solutions have the same regularity \cite{brezis.kinderlehrer,caffarelli.kinderlehrer,kinderlehrer.stampacchia}. This does not apply in the nonlinear setting, hence along the years, there have been intense research activities for the regularity of the obstacle problem in this direction.

In the case of standard growth conditions, Eleuteri and Passarelli di Napoli \cite{eleuteri.passarelli2018} proved that an extra differentiability of integer or fractional order of the gradient of the obstacle transfers to the gradient of the solutions, provided the partial map $x \mapsto \mathcal{A}(x,\xi)$ possesses a suitable differentiability property. 
\\Recently, Gavioli proved in \cite{gavioli1,gavioli2} that the weak differentiability of integer order of the partial map $x \mapsto \mathcal{A}(x, \xi)$ is a sufficient condition to prove that an extra differentiability of integer order of the gradient of the obstacle transfers to the gradient of the solutions to obstacle problems with $p,q$-growth conditions. The intermediate case of higher differentiability in the setting of variable exponents case has been carried out in the paper \cite{foralli.giliberti}. Furthermore, a higher fractional differentiability has been proved for solutions to double phase elliptic obstacle problems in \cite{zhang.zheng}. We remark that double phase elliptic obstacle problems can be obtained as a particular case of a functional satisfying our growth hypotheses, moreover the assumption made in \cite{zhang.zheng} on the coefficients of the operator $\mathcal{A}$ is stronger with respect to ours.
\\Here, we continue the study of the higher differentiablity properties of solutions to \eqref{1} in case of $p,q$-growth conditions. The novelty of this paper consists in assuming that both the gradient of the obstacle and the partial map $x \mapsto \mathcal{A}(x, \xi)$ belong to a suitable Sobolev class of fractional order.

Our analysis comes from the fact that the regularity of the solutions to the obstacle problem \eqref{1} is strictly connected to the analysis of the regularity of the solutions to partial differential equations of the form
$$\text{div} D_{\xi}F(x,Du)= \text{div}D_{\xi}F(x,D\psi),$$ 
whose higher differentiability properties have been widely investigated (see for example 
\cite{baison.clop2017,clop.giova.passarelli,giova,giova.passarelli,passarelli2014,
passarelli2014.potanal,passarelli2015}). We also notice that previous regularity results concerning local minimizers of integral functionals of the Calculus of Variations, under the assumption (A4), have been obtained by Kristensen and Mingione \cite{kristensen.mingione}.\\
In particular, our aim is to extend the higher differentiability results in \cite{eleuteri.passarelli2018} (see Theorems \ref{1thmp=q} and \ref{thmp=q} in Section \ref{secnot}) to the case of functionals with $p,q$--growth.
\begin{thm}\label{mainthm1}
Let $\A(x,\xi)$ satisfy (A1)-(A4) for exponents $2\leq p < q <\frac{n}{\alpha}< r$ such that
\begin{equation}\label{gap}
\dfrac{q}{p} <  1+ \dfrac{ \alpha}{n}- \frac{1}{r} .
\end{equation}
Let $u\in \mathcal{K}_\psi(\Omega)$ be the solution to the obstacle problem \eqref{1}. Then we have
\begin{equation}
D\psi \in B^\gamma_{2q-p,\infty,\textrm{loc}}(\Omega)\Rightarrow (\mu^2 +|Du|^2)^\frac{p-2}{4} Du \in B^\alpha_{2,\infty,\textrm{loc}}(\Omega),
\end{equation}
provided $0< \alpha < \gamma < 1$.
\end{thm}

We are also able to prove the following finite case.

\begin{thm}\label{mainthm2}
Let $\A(x,\xi)$ satisfy (A1)-(A3) and (A5) for exponents $2\leq p < q <\frac{n}{\alpha}< r$ such that
\begin{equation}\label{gap2}
\dfrac{q}{p} <  1+ \dfrac{ \min \{ \alpha, \gamma \}}{n}- \frac{1}{r} ,
\end{equation}
where $0< \gamma <1$.
Let $u\in \mathcal{K}_\psi(\Omega)$ be the solution to the obstacle problem \eqref{1}. Then we have
\begin{equation}
D\psi \in B^\gamma_{2q-p,\sigma,\textrm{loc}}(\Omega)\Rightarrow (\mu^2 +|Du|^2)^\frac{p-2}{4} Du \in B^{\min \{\alpha, \gamma  \}}_{2,\sigma,\textrm{loc}}(\Omega),
\end{equation}
provided $\sigma \leq \frac{2n}{n-2 \min \{ \alpha, \gamma \}}$.
\end{thm}

Existence of solutions to the obstacle problem \eqref{1} can be easily proved through classical results regarding variational inequalities, so in this paper we will mainly concentrate on the regularity results. The proof of Theorems \ref{mainthm1} and \ref{mainthm2} is achieved by means of difference quotient method, that is quite natural when trying to establish higher differentiabilty results and local gradient estimates (see for instance \cite{marcellini1, marcellini1991}). Here the difficulties come from the set of admissible test functions that have to take into account the presence of the obstacle. In order to overcome this issue, we consider difference quotient involving both the solution and the obstacle, so that the function satisfies the constraint of belonging to the admissible class $\mathcal{K}_{\psi}(\Omega)$.\\ Finally, we observe that the assumption \eqref{gap} is crucial for obtaining the conclusion of Theorem \ref{mainthm1}. This is the natural counterpart in the fractional setting of the corresponding inequality considered for the first time in \cite{eleuteri.marcellini.mascolo3}. Indeed, our estimate is sharp, that is for $\alpha=1$ we recover the result in \cite{eleuteri.marcellini.mascolo3}. In fact, when referring to $p,q$-growth conditions, in order to ensure the regularity of minima, the gap $q/p>1$ cannot differ too much from $1$ (see for instance the counterexamples \cite{fonseca.maly.mingione,giaquinta,marcellini1991}).
\\
The structure of this paper is the following. After recalling some notation and preliminary results in Section \ref{secnot}, we concentrate on proving our main results, Theorems \ref{mainthm1} and \ref{mainthm2}. In both cases, the strategy is to establish the a priori estimate for an approximating solution and then pass to the limit in the approximating problem. Therefore, we present our approximation results in 
Section \ref{secappr}, namely we are able to prove the existence of a sequence of functions with $p$-growth conditions that monotonically converges to our initial problems. In Section \ref{secthm1} we take care of Theorem \ref{mainthm1}. In particular, we derive the a priori estimates in Section \ref{ape} for an approximating problem satisfying standard growth conditions. Then, in Section \ref{secproof} we exploit the results of Sections \ref{secappr} and \ref{ape} and using compactness, strictly convexity and weak lower semi-continuity of functional $F$, we are able to prove Theorem \ref{mainthm1}. Eventually, in Section \ref{secthm2} we prove Theorem \ref{mainthm2}, focusing on the a priori estimate, since the limit procedure works exactly in the same way as for the previous result.

\section{Notations and preliminary results}
\label{secnot}
In what follows, $B(x,r)=B_{r}(x)= \{ y \in \mathbb{R}^{n} : |y-x | < r  \}$ will denote the ball centered at $x$ of radius $r$. We shall omit the dependence on the center and on the radius when no confusion arises. For a function $u \in L^{1}(B)$, the symbol
\begin{center}
$\displaystyle\fint_{B} u(x) dx = \dfrac{1}{|B|} \displaystyle\int_{B} u(x) dx$.
\end{center}
will denote the integral mean of the function $u$ over the set $B$.

It is convenient to introduce an auxiliary function
\begin{center}
$V_{p}(\xi)=(\mu^{2}+|\xi|^{2})^\frac{p-2}{4} \xi$
\end{center}
defined for all $\xi\in \mathbb{R}^{n}$. One can easily check that, for $p \geq 2$, there exists an absolute constant $c$ such that
\begin{gather}
|\xi|^p \leq c |V_p(\xi)|^2. \label{Vp}
\end{gather}
For the auxiliary function $V_{p}$, we recall the following estimate (see the proof of \cite[Lemma 8.3]{giusti}): 
\begin{lem}\label{D1}
Let $1<p<+\infty$. There exists a constant $c=c(n,p)>0$ such that
\begin{center}
$c^{-1}(\mu^{2}+|\xi|^{2}+|\eta|^{2})^{\frac{p-2}{2}} \leq \dfrac{|V_{p}(\xi)-V_{p}(\eta)|^{2}}{|\xi-\eta|^{2}} \leq c(\mu^{2}+|\xi|^{2}+|\eta|^{2})^{\frac{p-2}{2}} $
\end{center}
for any $\xi, \eta \in \mathbb{R}^{n}$.
\end{lem}

Now we state a well-known iteration lemma (see \cite{giusti} for the proof).
\begin{lem}\label{lm2}
Let $\Phi  :  [\frac{R}{2},R] \rightarrow \mathbb{R}$ be a bounded nonnegative function, where $R>0$. Assume that for all $\frac{R}{2} \leq r < s \leq R$ it holds
$$\Phi (r) \leq \theta \Phi(s) +A + \dfrac{B}{(s-r)^2}+ \dfrac{C}{(s-r)^{\gamma}}$$
where $\theta \in (0,1)$, $A$, $B$, $C \geq 0$ and $\gamma >0$ are constants. Then there exists a constant $c=c(\theta, \gamma)$ such that
$$\Phi \biggl(\dfrac{R}{2} \biggr) \leq c \biggl( A+ \dfrac{B}{R^2}+ \dfrac{C}{R^{\gamma}}  \biggr).$$
\end{lem}

\subsection{Besov-Lipschitz spaces}
\label{secbesov}
Let $v:\mathbb{R}^{n} \rightarrow \mathbb{R}$ be a function. As in \cite[Section 2.5.12]{haroske}, given $0< \alpha <1$ and $1 \leq p,q < \infty$, we say that $v$ belongs to the Besov space $B^{\alpha}_{p,q}(\mathbb{R}^{n})$ if $v \in L^{p}(\mathbb{R}^{n})$ and
\begin{center}
$\Vert v \Vert_{B^{\alpha}_{p,q}(\mathbb{R}^{n})} = \Vert v \Vert_{L^{p}(\mathbb{R}^{n})} + [v]_{B^{\alpha}_{p,q}(\mathbb{R}^{n})} < \infty$,
\end{center}
where
\begin{center}
$[v]_{B^{\alpha}_{p,q}(\mathbb{R}^{n})} =  \biggl( \displaystyle\int_{\mathbb{R}^{n}} \biggl( \displaystyle\int_{\mathbb{R}^{n}} \dfrac{|v(x+h)-v(x)|^{p}}{|h|^{\alpha p}} dx \biggr)^{\frac{q}{p}}  \dfrac{dh}{|h|^{n}} \biggr)^{\frac{1}{q}}  < \infty$.
\end{center}
Equivalently, we could simply say that $v \in L^{p}(\mathbb{R}^{n})$ and $\frac{\tau_{h}{v}}{|h|^{\alpha}} \in L^{q}\bigl( \frac{dh}{|h|^{n}}; L^{p}(\mathbb{R}^{n}) \bigr)$. As usual, if one integrates for $h \in B(0, \delta)$ for a fixed $\delta >0$ then an equivalent norm is obtained, because
\begin{center}
$\biggl( \displaystyle\int_{\{|h| \geq \delta\}} \biggl( \displaystyle\int_{\mathbb{R}^{n}} \dfrac{|v(x+h)-v(x)|^{p}}{|h|^{\alpha p}} dx \biggr)^{\frac{q}{p}}  \dfrac{dh}{|h|^{n}} \biggr)^{\frac{1}{q}} \leq c(n, \alpha,p,q, \delta) \Vert v \Vert_{L^{p}(\mathbb{R}^{n})} $.
\end{center}
Similarly, we say that $v \in B^{\alpha}_{p,\infty}(\mathbb{R}^{n})$ if $v \in L^{p}(\mathbb{R}^{n})$ and
\begin{center}
$[v]_{B^{\alpha}_{p, \infty}(\mathbb{R}^{n})} =  \displaystyle\sup_{h \in \mathbb{R}^{n}} \biggl( \displaystyle\int_{\mathbb{R}^{n}} \dfrac{|v(x+h)-v(x)|^{p}}{|h|^{\alpha p}} dx \biggr)^{\frac{1}{p}} < \infty $.
\end{center}
Again, one can simply take supremum over $|h| \leq \delta$ and obtain an equivalent norm. By construction, $B^{\alpha}_{p, q}(\mathbb{R}^{n}) \subset L^{p}(\mathbb{R}^{n})$. One also has the following version of Sobolev embeddings (a proof can be found at \cite[Proposition 7.12]{haroske}).
\begin{lem}\label{3.1}
Suppose that $0 < \alpha <1$.
\\ (a) If $1 < p < \frac{n}{\alpha}$ and $1 \leq q \leq p^{*}_{\alpha} = \frac{np}{n- \alpha p}$, then there is a continuous embedding $B^{\alpha}_{p, q}(\mathbb{R}^{n}) \subset L^{p^{*}_{\alpha}}(\mathbb{R}^{n})$.
\\ (b) If $p = \frac{n}{\alpha}$ and $1 \leq q \leq \infty$, then there is a continuous embedding $B^{\alpha}_{p, q}(\mathbb{R}^{n}) \subset BMO(\mathbb{R}^{n})$,
\\ where $BMO$ denotes the space of functions with bounded mean oscillations \emph{\cite[Chapter 2]{giusti}}.
\end{lem}
For further needs, we recall the following inclusions (\cite[Proposition 7.10 and Formula (7.35)]{haroske}).
\begin{lem}\label{3.2}
Suppose that $0 < \beta < \alpha < 1$.
\\ (a) If $1 < p < \infty$ and $1 \leq q \leq r \leq \infty$, then $B^{\alpha}_{p, q}(\mathbb{R}^{n}) \subset B^{\alpha}_{p, r}(\mathbb{R}^{n})$.
\\ (b) If $1 < p < \infty$ and $1 \leq q , r \leq \infty$, then $B^{\alpha}_{p, q}(\mathbb{R}^{n}) \subset B^{\alpha}_{p, r}(\mathbb{R}^{n})$.
\\ (c) If $1 \leq q \leq \infty$, then $B^{\alpha}_{\frac{n}{\alpha}, q}(\mathbb{R}^{n}) \subset B^{\beta}_{\frac{n}{\beta}, q}(\mathbb{R}^{n})$.
\end{lem}

Given a domain $\Omega \subset \mathbb{R}^{n}$, we say that $v$ belongs to the local Besov space $ B^{\alpha}_{p, q,loc}$ if $\varphi \ v \in B^{\alpha}_{p, q}(\mathbb{R}^{n})$ whenever $\varphi \in \mathcal{C}^{\infty}_{c}(\Omega)$. It is worth noticing that one can prove suitable version of Lemma \ref{3.1} and Lemma \ref{3.2}, by using local Besov spaces.

The following Lemma can be found in \cite{baison.clop2017}.
\begin{lem}
A function $v \in L^{p}_{loc}(\Omega)$ belongs to the local Besov space $B^{\alpha}_{p,q,loc}$ if, and only if,
\begin{center}
$\biggl\Vert \dfrac{\tau_{h}v}{|h|^{\alpha}} \biggr\Vert_{L^{q}\bigl(\frac{dh}{|h|^{n}};L^{p}(B)\bigr)}<  \infty$
\end{center}
for any ball $B\subset2B\subset\Omega$ with radius $r_{B}$. Here the measure $\frac{dh}{|h|^n}$ is restricted to the ball $B(0,r_B)$ on the h-space.
\end{lem}

It is known that Besov-Lipschitz spaces of fractional order $\alpha \in (0,1)$ can be characterized in pointwise terms. Given a measurable function $v:\mathbb{R}^{n} \rightarrow \mathbb{R}$, a \textit{fractional $\alpha$-Hajlasz gradient for $v$} is a sequence $\{g_{k}\}_{k}$ of measurable, non-negative functions $g_{k}:\mathbb{R}^{n} \rightarrow \mathbb{R}$, together with a null set $N\subset\mathbb{R}^{n}$, such that the inequality 
\begin{center}
$|v(x)-v(y)|\leq (g_{k}(x)+g_{k}(y))|x-y|^{\alpha}$
\end{center} 
holds whenever $k \in \mathbb{Z}$ and $x,y \in \mathbb{R}^{n}\setminus N$ are such that $2^{-k} \leq|x-y|<2^{-k+1}$. We say that $\{g_{k}\}_{k} \in l^{q}(\mathbb{Z};L^{p}(\mathbb{R}^{n}))$ if
\begin{center}
$\Vert \{g_{k}\}_{k} \Vert_{l^{q}(L^{p})}=\biggl(\displaystyle\sum_{k \in \mathbb{Z}}\Vert g_{k} \Vert^{q}_{L^{p}(\mathbb{R}^{n})} \biggr)^{\frac{1}{q}}<  \infty.$
\end{center} 

The following result was proved in \cite{koskela}.
\begin{thm}
Let $0< \alpha <1,$ $1 \leq p < \infty$ and $1\leq q \leq \infty $. Let $v \in L^{p}(\mathbb{R}^{n})$. One has $v \in B^{\alpha}_{p,q}(\mathbb{R}^{n})$ if, and only if, there exists a fractional $\alpha$-Hajlasz gardient $\{g_{k}\}_{k} \in l^{q}(\mathbb{Z};L^{p}(\mathbb{R}^{n}))$ for $v$. Moreover,
\begin{center}
$\Vert v \Vert_{B^{\alpha}_{p,q}(\mathbb{R}^{n})}\simeq \inf \Vert \{g_{k}\}_{k} \Vert_{l^{q}(L^{p})},$
\end{center}
where the infimum runs over all possible fractional $\alpha$-Hajlasz gradients for $v$.
\end{thm}

\subsection{Difference quotient}
\label{secquo}
We recall some properties of the finite difference quotient operator that will be needed in the sequel. Let us recall that, for every function $F:\mathbb{R}^{n}\rightarrow \mathbb{R}$ the finite difference operator is defined by
\begin{center}
$\tau_{s,h}F(x)=F(x+he_{s})-F(x)$
\end{center}
where $h \in \mathbb{R}^{n}$, $e_{s}$ is the unit vector in the $x_{s}$ direction and $s \in \{1,...,n\}$.
\\We start with the description of some elementary properties that can be found, for example, in \cite{giusti}.
\begin{prop}
Let $F$ and $G$ be two functions such that $F,G \in W^{1,p}(\Omega)$, with $p \geq1$, and let us consider the set
\begin{center}
$\Omega_{|h|} = \{ x \in \Omega : \text{dist}(x,\partial \Omega)> |h|  \}$.
\end{center}
Then
\\(i) $\tau_{h}F \in W^{1,p}(\Omega_{|h|})$ and 
\begin{center}
$D_{i}(\tau_{h}F)=\tau_{h}(D_{i}F)$.
\end{center}
(ii) If at least one of the functions $F$ or $G$ has support contained in $\Omega_{|h|}$, then
\begin{center}
$$\displaystyle\int_{\Omega}F \tau_h G dx = \displaystyle\int_{\Omega} G \tau_{-h}F dx.$$
\end{center}
(iii) We have $$\tau_h (FG)(x)= F(x+h)\tau_h G(x)+G(x) \tau_h F(x).$$
\end{prop}
The next result about finite difference operator is a kind of integral version of Lagrange Theorem.
\begin{lem}\label{ldiff}
If $0<\rho<R,$ $|h|<\frac{R-\rho}{2},$ $1<p<+\infty$ and $F,\ DF \in L^{p}(B_{R})$, then
\begin{center}
$\displaystyle\int_{B_{\rho}} |\tau_{h}F(x)|^{p} dx \leq c(n,p)|h|^{p} \displaystyle\int_{B_{R}} |DF(x)|^{p}dx$.
\end{center}
Moreover,
\begin{center}
$\displaystyle\int_{B_{\rho}} |F(x+h)|^{p} dx \leq  \displaystyle\int_{B_{R}} |F(x)|^{p}dx$.
\end{center}
\end{lem}

We conclude this subsection recalling the following Lemma (see \cite{kristensen.mingione}), which can be seen as a consequence of Lemmas \ref{3.1} and \ref{3.2}.
\begin{lem}\label{Besovlem}
Let $F \in L^2(B_R)$. Suppose that there exist $\rho \in (0,R)$, $0 < \alpha < 1$ and $M >0$ such that $$ \displaystyle\sum_{s=1}^{n} \displaystyle\int_{B_{\rho}}|\tau_{s,h}F(x)|^2 dx \leq M^2 |h|^{2 \alpha}, $$ for every $h$ such that $h < \frac{R-\rho}{2}$. Then $F \in L^{\frac{2n}{n-2 \beta}}(B_{\rho})$ for every $\beta \in (0, \alpha)$ and $$\Vert  F \Vert _{L^{\frac{2n}{n-2 \beta}}(B_{\rho})} \leq c (M+ \Vert F \Vert _{L^2(B_R)}),  $$ with $c=c(n,N,R,\rho, \alpha, \beta)$.
\end{lem}

\subsection{Preliminary results on standard growth conditions}
For sake of clarity, we would like to recall the following regularity results (see \cite{eleuteri.passarelli2018} for the proof), which will be used in order to prove Theorems \ref{mainthm1} and \ref{mainthm2}.

\begin{thm}\label{1thmp=q}
Assume that $\mathcal{A}(x,\xi)$ satisfies (A1)-(A3) for an exponent $2 \leq p=q < \frac{n}{\alpha}$ and let $u \in \mathcal{K}_{\psi}(\Omega)$ be the solution to the obstacle problem \eqref{1}. If there exists a sequence of measurable non-negative functions $g_k \in L^{\frac{n}{\alpha}}_{\text{loc}}(\Omega)$ such that
$$\displaystyle\sum_{k=1}^{\infty} \Vert g_k \Vert^{\sigma}_{L^{\frac{n}{\alpha}}(\Omega)} < \infty,$$
and at the same time
$$|\mathcal{A}(x,\xi)-\mathcal{A}(y, \xi)| \leq |x-y|^{\alpha} (g_k(x)+g_k(y)) (\mu^2 +|\xi|^2)^{\frac{p-1}{2}} ,$$
for a.e. $x,y \in \Omega$ such that $2^{-k} \text{diam}(\Omega) \leq |x-y| < 2^{-k+1}\text{diam}(\Omega)$ and for every $\xi \in \mathbb{R}^n$, then the following implication
$$  D\psi \in B^\gamma_{p,\sigma,\textrm{loc}}(\Omega)\Rightarrow (\mu^2 +|Du|^2)^\frac{p-2}{4} Du \in B^{\min\{ \alpha, \gamma \}}_{2,\sigma,\textrm{loc}}(\Omega),$$
holds, provided $\sigma \leq p_{\gamma}^*= \frac{np}{n-\gamma p}$.
\end{thm}
In the case of a regularity of the type $B^{\alpha}_{p,\infty}$, which is the weakest one in the scale of Besov spaces, both on the coefficients and on the gradient of the obstacle, we have the following

\begin{thm}\label{thmp=q}
Assume that $\mathcal{A}(x,\xi)$ satisfies (A1)-(A3) for an exponent $2 \leq p=q < \frac{n}{\alpha}$ and let $u \in \mathcal{K}_{\psi}(\Omega)$ be the solution to the obstacle problem \eqref{1}. If there exists a non-negative function $k \in L^{\frac{n}{\alpha}}_{\text{loc}}(\Omega)$ such that
$$|\mathcal{A}(x,\xi)-\mathcal{A}(y, \xi)| \leq |x-y|^{\alpha} (k(x)+k(y)) (\mu^2 +|\xi|^2)^{\frac{p-1}{2}} ,$$
for a.e. $x,y \in \Omega$ and for every $\xi \in \mathbb{R}^n$, then the following implication
$$  D\psi \in B^\gamma_{p,\infty,\textrm{loc}}(\Omega)\Rightarrow (\mu^2 +|Du|^2)^\frac{p-2}{4} Du \in B^\alpha_{2,\infty,\textrm{loc}}(\Omega),$$
holds, provided $0< \alpha < \gamma < 1$.
\end{thm}

\section{Approximation results}
\label{secappr}
We here collect some results which will be used to prove the passage to the limit in Theorems \ref{mainthm1} and \ref{mainthm2}.\\
We first recall the following Theorem, whose complete version can be found in \cite{cupini.guidorzi.mascolo2003} and which will be used to prove Lemma \ref{apprlem1}.

\begin{thm}\label{thmdefF}
Let $F: \Omega \times \R^n \rightarrow [0,+\infty), F= F(x,\xi)$, be a Carath\'{e}odory function. Then, assumptions (F2) and (F3) imply that
there exist $c_0(p,q,\nu,R, l, L),\ c_1(p, \nu)>0$ and a Carath\'{e}odory function $g:\Omega \times \R^n \rightarrow [-c_0,+\infty)$ s.t. for a.e. $x\in \Omega$ and every $\xi \in \R^n$,
\[ F(x,\xi) = c_1(\mu^2 +|\xi|^2)^\frac{p}{2}+ g(x,\xi).\]
\end{thm}

In the next lemma, we adapt a well known approximation result, which can be found in \cite{cupini.guidorzi.mascolo2003}, to the case when the map $x \mapsto D_{\xi}F(x, \xi)$ has a Besov regularity.
\begin{lem}\label{apprlem1}
Let $F : \Omega \times \R^n \rightarrow [0,+\infty), F = F(x, \xi)$, be a Carath\'{e}odory function,
convex with respect to $\xi$, satisfying assumptions (F1), (F2), (F3) and (F5). Then there exists
a sequence $(F_j)$ of Carath\'{e}odory functions $F_j: \Omega \times \R^n \rightarrow [0,+\infty)$, convex with respect to the last variable, monotonically convergent to $F$, such that
\begin{itemize}
\item[(i)]for a.e. $x \in \Omega$ and every $\xi \in \R^n$, 
$F_j(x,\xi) = \tilde{F}_j (x,|\xi|)$,
\item[(ii)]for a.e. $x \in \Omega$, for every 	$\xi \in \R^n$ and for every $j$, 
$F_j(x,\xi)\leq F_{j+1}(x, \xi) \leq F(x,\xi)$,
\item[(iii)]for a.e. $x \in \Omega$ and every $\xi \in \R^n$, we have
$\langle D_{\xi\xi} F_j(x,\xi)\lambda, \lambda\rangle \geq \bar{\nu}(\mu^2+|\xi|^2)^\frac{p-2}{2}|\lambda|^2$,
with $\bar{\nu}$ depending only on $p$ and $\nu$,
\item[(iv)] for a.e. $x\in \Omega$ and for every $\xi \in \R^n$, there exist $L_1$, independent of $j$, and $\bar{L}_1$, depending on $j$, such that
\begin{align*}
& 1/L_1(\mu + |\xi|)^p \leq F_j(x,\xi)\leq L_1(\mu + |\xi|)^q,\\
&F_j (x,\xi)\leq \bar{L}_1(j)(\mu + |\xi|)^p,
\end{align*} 
\item[(v)]there exists a constant $C(j) > 0$ such that
\begin{align*}
&|D_\xi F_j(x,\xi)-D_\xi F_j(y,\xi)| \leq |x-y|^{\alpha} (k(x)+k(y))(\mu^2 +|\xi|^2)^\frac{q-1}{2},\\
&|D_\xi F_j(x,\xi)-D_\xi F_j(y,\xi)| \leq  C(j)|x-y|^{\alpha} (k(x)+k(y))(\mu^2 +|\xi|^2)^\frac{p-1}{2}
\end{align*}
for a.e. $x,y \in \Omega$ and for every $\xi \in \R^n$.
\end{itemize}
\end{lem}

\begin{proof}
According to Theorem \ref{thmdefF}, which holds under hypotheses (F2) and (F3), there exist the positive constants
$c_0 = c_0(p,q,\nu,R,l,L)$ and $c_1 =c_1(p,\nu)$ and a function $g: \Omega \times \R^n \rightarrow [- c_0, +\infty)$ s.t.
\begin{equation}\label{lemdefF}
F(x, \xi) = c_1(\mu^2 + |\xi|^2)^\frac{p}{2} + g(x,\xi)
\end{equation}
with $g$ convex. Moreover there exists $\tilde{g}: \Omega \times [0,+\infty) \rightarrow [- c_0,+\infty)$ s.t.
$\tilde{g}(x,|\xi|) = g(x,\xi)$ for any $\xi \in \R^n$. Since $n\geq 2$, for a.e. $x\in \Omega$, $t\mapsto \tilde{g}(x,t)$ is convex and increasing. For any $j \in \N$, we might then define $\tilde{g}_j: \Omega \times [0,+\infty) \rightarrow [-c_0, +\infty)$ as
\begin{align*}
&\tilde{g}_j(x,t) = \tilde{g}(x,t) \qquad \forall (x,t) \in \Omega \times [0,j],\\
&\tilde{g}_j(x,t) = \tilde{g}(x,j) + D_t \tilde{g}(x,j)(t-j) \qquad \forall (x,t) \in \Omega \times (j,\infty)
\end{align*}
We notice that, by definition, for a.e. $x\in \Omega, t\mapsto \tilde{g}_j(x,t)$ is convex and increasing in $[0,+\infty)$ and $\tilde{g}_j(x,t) \leq\tilde{g}_{j+1}(x,t) \leq \tilde{g}(x,t)$.
Combining assumption (F2), the definition of $\tilde{g}_j(x,t)$ and \eqref{lemdefF}, we infer
\begin{align} \nonumber
&\tilde{g}_j(x,t) \leq l (\mu + t)^q,\\ \label{gtil}
&\tilde{g}_j(x,t) \leq c(q,l,j) (\mu + t)^p.
\end{align}
We now want to show that $D_t \tilde{g}_j$ has a (F5)-type growth. It is easy to see that $D_t \tilde{g}_j(x,t) = D_t \tilde{g}(x,j)$ for $t\geq j$. In particular, assumption (F5) yields 
$|D_t \tilde{g}(x,j)-D_t \tilde{g}(y,j)| \leq |x-y|^\alpha(k(x)+k(y))(\mu + j)^{q-1}$. Hence, for a.e. $x \in \Omega$ and every $t> 0$,
\begin{align}\label{disgradiente}
&|D_t \tilde{g}(x,t)-D_t \tilde{g}(y,t)| \leq |x-y|^\alpha(k(x)+k(y))(\mu + t)^{q-1}.
\end{align}
Moreover, for $t \leq j$, according to \eqref{lemdefF} and \eqref{disgradiente}, we obtain  
\begin{align*}
|D_t \tilde{g}(x,t)-D_t \tilde{g}(y,t)| 
\leq & |x-y|^\alpha(k(x)+k(y))(\mu + t)^{p-1} (\mu +t)^{q-p} \\
\leq &  |x-y|^\alpha(k(x)+k(y))(\mu + t)^{p-1} (\mu +j)^{q-p} \\
\leq & c(j) |x-y|^\alpha(k(x)+k(y))(\mu + t)^{p-1}.
\end{align*}
On the other hand, in the same way, for $t > j$, we get
\begin{align*}
|D_t \tilde{g}(x,t)-D_t \tilde{g}(y,t)| 
\leq & |x-y|^\alpha(k(x)+k(y))(\mu + j)^{p-1} (\mu +j)^{q-p} \\
\leq & |x-y|^\alpha(k(x)+k(y))(\mu + t)^{p-1} (\mu +j)^{q-p} \\
\leq & c(j) |x-y|^\alpha(k(x)+k(y))(\mu + t)^{p-1}.
\end{align*}
Eventually, for any $j$, we define $g_j: \Omega \times \R^n \rightarrow [-c_0, +\infty)$ as
\begin{align*}
g_j(x,\xi)= \tilde{g}_j(x,|\xi|). 
\end{align*}
Statements $(i),(ii),(iii),(v)$ directly follow by setting $F_j: \Omega \times \R^n \rightarrow[0,+\infty)$
\[F_j(x, \xi) := c_1(\mu^2 + |\xi|^2)^\frac{p}{2} + g_j(x,\xi).\]
Property $(iv)$ is obtained  combining \eqref{lemdefF} with \eqref{gtil} and the definition of $F_j$.
\end{proof}

\begin{rmk}\label{remappr}
It is worth noting that an analogous version of Lemma \ref{apprlem1} can be proved similarly, supposing (F6) instead of (F5). In particular, statement $(v)$ would change as follows.
\begin{itemize}
\item[(v)]There exists a constant $C(j) > 0$ such that
\begin{align*}
&|D_\xi F_j(x,\xi)-D_\xi F_j(y,\xi)| \leq |x-y|^{\alpha} (g_k(x)+g_k(y))(\mu^2 +|\xi|^2)^\frac{q-1}{2},\\
&|D_\xi F_j(x,\xi)-D_\xi F_j(y,\xi)| \leq  C(j)|x-y|^{\alpha} (g_k(x)+g_k(y))(\mu^2 +|\xi|^2)^\frac{p-1}{2}
\end{align*}
for a.e. $x,y \in \Omega$ such that $2^{-k} \text{diam}(\Omega) \leq |x-y| < 2^{-k+1}\text{diam}(\Omega)$ and for every $\xi \in \R^n$.
\end{itemize}
\end{rmk}

\section{Proof of Theorem \ref{mainthm1}}\label{secthm1}
In order to prove Theorem \ref{mainthm1}, in Section \ref{aprioriestimate}, we derive a suitable a priori estimate for minimizers of obstacle problems with $p$-growth conditions, while in Section \ref{passagetothelimit}, we conclude showing that the a priori estimate is preserved when passing to the limit.

\subsection{A priori estimate}\label{aprioriestimate}
Let us consider
\begin{gather}\label{regprob}
\min \biggl\{ \displaystyle\int_{\Omega} F_j(x,Dw)dx \ : \ w \in \mathcal{K}_{\psi}(\Omega)  \biggr\},
\end{gather}
where $F_j: \Omega \times \R^n \rightarrow [0,+\infty)$, $F_j=F_j(x, \xi)$, was set in Lemma \ref{apprlem1}.
\\Setting 
\begin{gather}
\mathcal{A}_j(x,\xi)= D_{\xi}F_j(x,\xi), \notag
\end{gather}
one can easily check that $\mathcal{A}_j$ satisfies (A1)--(A4) and the following assumptions:
\begin{gather}
|\mathcal{A}_j(x,\xi)| \leq l_1(j)(\mu^2 +|\xi|^2)^{\frac{p-1}{2}} \label{K2} \\
|\mathcal{A}_j(x,\xi)-\mathcal{A}_j(x,\eta)|\leq L_1(j)|\xi-\eta| (\mu^2+|\xi|^{2}+|\eta|^{2})^{\frac{p-2}{2}} \label{K1}\\
|\mathcal{A}_j(x,\xi)-\mathcal{A}_j(y, \xi)| \leq \Theta(j)|x-y|^{\alpha} (k(x)+k(y)) (\mu^2 +|\xi|^2)^{\frac{p-1}{2}} \label{K3}
\end{gather}
for a.e. $x,y \in \Omega$, for every $\xi, \eta \in \mathbb{R}^n$.
It is well known that $u_j \in \mathcal{K}_{\psi}(\Omega)$ is a minimizer of problem \eqref{regprob} if, and only if, the following variational inequality holds 
\begin{align}\label{varineq2}
\displaystyle\int_{\Omega} \langle \mathcal{A}_j(x,Du_j)  ,D(\varphi-u_j) \rangle dx \geq 0 , \quad \forall \varphi\in \mathcal{K}_{\psi}(\Omega).
\end{align}

The following result holds:
\begin{thm}\label{approximation}
Let $\mathcal{A}_j(x,\xi)$ satisfy (A1)--(A4) and \eqref{K1}-- \eqref{K3} for exponents $2\leq p < q <\frac{n}{\alpha}< r$ satisfying \eqref{gap}.
Let $u_j\in \mathcal{K}_\psi(\Omega)$ be the solution to the obstacle problem \eqref{varineq2}. Suppose that $k \in L^{r}_{loc}(\Omega)$ and $D \psi \in B^{\gamma}_{2q-p, \infty,loc}(\Omega)$, for $0 < \alpha < \gamma <1$. Then, the following estimate
\begin{gather}
\displaystyle\int_{B_{R/4}} |\tau_h V_p(Du_j)|^2 dx \leq C |h|^{2\alpha} \biggl\{  \ibR (1+ |Du_j|^p)dx+ \Vert D \psi \Vert_{B^{\gamma}_{2q-p, \infty}(B_R)} \biggl\}^{\kappa},
\end{gather}
holds for all balls $B_{R/4} \subset B_R \Subset \Omega$, for positive constants $C := C(R,n,p,q,r, \beta)$, $\kappa:= \kappa (n,p,q,r, \beta)$, both independent of $j$, and for some $0<\beta < \alpha$.
\end{thm}\label{ape}

\proof
\noindent We start by observing that, since $p < 2q-p $, we have
$$D \psi \in B^{\gamma}_{2q-p,\infty,loc}(\Omega) \Rightarrow D \psi \in B^{\gamma}_{p,\infty,loc}(\Omega),$$
thus an application of Theorem \ref{thmp=q} implies 
$$(\mu^2+ |Du_j|^2)^{\frac{p-2}{4}}Du_j \in B^{\alpha}_{2,\infty,loc}(\Omega),$$
which yields, by applying Lemma \ref{Besovlem},
$$Du_j \in L^{\frac{np}{n-2 \beta}}_{loc}(\Omega),$$
for all $0< \beta < \alpha$.
Thus, the integral 
$$\displaystyle\int_{\Omega'}(1+ |Du_j|)^{\frac{np}{n-2\beta}}dx$$
is finite, for every $\Omega' \Subset \Omega$ and $\beta \in (0, \alpha)$.
\\In the sequel we will profusely use the following inequality:
\begin{equation}\label{dis}
2q-p \leq \dfrac{r(2q-p)}{r-2} \leq \dfrac{np}{n-2 \beta} ,
\end{equation}
for $\beta \in (\frac{\alpha n r}{nr+2(\alpha r -n)}, \alpha)$. The first part of inequality \eqref{dis} is trivial, while the second part comes from \eqref{gap}. Namely,
\begin{align*}
\dfrac{r(2q-p)}{r-2} \leq  \dfrac{np}{n- 2 \beta} \Leftrightarrow \dfrac{q}{p} \leq \dfrac{nr-n- \beta r}{r (n- 2 \beta)}
\end{align*}
and
\begin{align*}
1+ \dfrac{\alpha}{n}-\dfrac{1}{r} < \dfrac{nr-n- \beta r}{r (n- 2 \beta)} \Leftrightarrow \beta > \frac{\alpha n r}{nr+2(\alpha r -n)}.
\end{align*}

Fix $0 < \frac{R}{4} < \rho < s < t < t' < \frac{R}{2} $ such that $B_{R} \Subset \Omega$ and a cut-off function $\eta \in \mathcal{C}_0^1(B_t)$ such that $0 \leq \eta \leq 1$, $\eta =1$ on $B_s$, $|D \eta | \leq \frac{C}{t-s}$.
\\Now, for $|h| < \frac{R}{4}$, we consider functions
\begin{gather}
v_{1}(x)= \eta^{2}(x) [(u_j-\psi)(x+h)-(u_j-\psi)(x)] \notag 
\end{gather}
and
\begin{gather}
v_{2}(x)= \eta^{2}(x-h) [(u_j-\psi)(x-h)-(u_j-\psi)(x)]. \notag 
\end{gather}
Then
\begin{equation}
\varphi_1(x)=u_j(x)+tv_1 (x), \label{2:2} 
\end{equation}
\begin{equation}
\varphi_2(x)=u_j(x)+tv_2(x) \label{2:3}
\end{equation}
are admissible test functions for all $t \in [0,1)$.
\\Inserting \eqref{2:2} and \eqref{2:3} in \eqref{varineq2}, we obtain
\begin{align}
\displaystyle\int_{\Omega}  \langle \mathcal{A}_j(x,Du_j), & D(\eta^2 \tau_h(u_j- \psi)) \rangle dx + \displaystyle\int_{\Omega} \langle \mathcal{A}_j(x,Du_j), D(\eta^2(x-h) \tau_{-h}(u_j- \psi)) \rangle dx \geq 0 \label{2:4}
\end{align}
By means of a simple change of variable, we can write the second integral on the left hand side of the previous inequality as follows
\begin{align}
\displaystyle\int_{\Omega} \langle \mathcal{A}_j(x+h,Du_j(x+h)), D(-\eta^2 \tau_h(u_j- \psi)) \rangle dx \label{2:5}
\end{align}
and so inequality \eqref{2:4} becomes
\begin{align}
\displaystyle\int_{\Omega} \langle \mathcal{A}_j(x+h,Du_j(x+h))-\mathcal{A}_j(x,Du_j(x)), D(\eta^2 \tau_h(u_j- \psi)) \rangle dx 
\leq 0 \label{2:6}
\end{align}
We can write previous inequality as follows
\begin{align}
0 \geq & \displaystyle\int_{\Omega} \langle \mathcal{A}_j(x+h,Du_j(x+h))-\mathcal{A}_j(x+h,Du_j(x)),\eta^{2}D\tau_{h}u_j \rangle dx\notag\\
  &-\displaystyle\int_{\Omega} \langle \mathcal{A}_j(x+h,Du_j(x+h))-\mathcal{A}_j(x+h,Du_j(x)),\eta^{2}D\tau_{h}\psi \rangle dx\notag\\
  &+\displaystyle\int_{\Omega} \langle \mathcal{A}_j(x+h,Du_j(x+h))-\mathcal{A}_j(x+h,Du_j(x)),2\eta D \eta\tau_{h}(u_j-\psi) \rangle dx\notag\\
  &+\displaystyle\int_{\Omega} \langle \mathcal{A}_j(x+h,Du_j(x))-\mathcal{A}_j(x,Du_j(x)),\eta^{2}D\tau_{h}u_j \rangle dx\notag\\
  &-\displaystyle\int_{\Omega} \langle \mathcal{A}_j(x+h,Du_j(x))-\mathcal{A}_j(x,Du_j(x)),\eta^{2}D\tau_{h}\psi \rangle dx\notag\\
  &+\displaystyle\int_{\Omega} \langle \mathcal{A}_j(x+h,Du_j(x))-\mathcal{A}_j(x,Du_j(x)),2\eta D \eta\tau_{h}(u_j-\psi) \rangle dx \notag\\
 =:& I_{1}+I_{2}+I_{3}+I_{4}+I_{5}+I_{6}, \label{2:7}
\end{align}
that yields
\begin{align}
I_1 \leq & |I_2| + |I_3| + |I_4| + |I_5| + |I_6|  \label{2:8}
\end{align}
The ellipticity assumption (A2) implies
\begin{align}
I_{1} \geq \nu \displaystyle\int_{\Omega}  \eta^{2} |\tau_{h}Du_j|^{2}(\mu^2 + |Du_j(x+h)|^{2}+|Du_j(x)|^{2})^{\frac{p-2}{2}} dx \label{I1}
\end{align}
From the growth condition (A3), Young's and H\"{o}lder's inequalities and assumption on $D \psi$, we get
\begin{align}
|I_{2}|\leq & L \displaystyle\int_{\Omega} \eta^{2} |\tau_{h}Du_j|(\mu^2 + |Du_j(x+h)|^{2}+|Du_j(x)|^{2})^{\frac{q-2}{2}}|\tau_{h}D \psi| dx \notag \\
\leq & \varepsilon\displaystyle\int_{\Omega} \eta^2 |\tau_hDu_j|^2 (\mu^2+|Du_j(x+h)|^2+|Du_j(x)|^2)^{\frac{p-2}{2}}dx \notag\\
&+ C_{\varepsilon}(L) \displaystyle\int_{\Omega} \eta^2 |\tau_h D \psi|^2 (\mu^2+|Du_j(x+h)|^2+|Du_j(x)|^2)^{\frac{2q-p-2}{2}} dx\notag\\
\leq & \varepsilon \displaystyle\int_{\Omega} \eta^2 |\tau_hDu_j|^2 (\mu^2+|Du_j(x+h)|^2+|Du_j(x)|^2)^{\frac{p-2}{2}}dx \notag\\
&+ C_{\varepsilon}(L) \biggl( \displaystyle\int_{B_t}|\tau_h D \psi|^{2q-p}dx \biggr)^{\frac{2}{2q-p}}  \biggl( \displaystyle\int_{B_{t'}}(1+|D u_j|)^{2q-p}dx \biggr)^{\frac{2q-p-2}{2q-p}}\notag\\
\leq & \varepsilon \displaystyle\int_{\Omega} \eta^2 |\tau_hDu_j|^2 (\mu^2+|Du_j(x+h)|^2+|Du_j(x)|^2)^{\frac{p-2}{2}}dx \notag\\
&+ C_{\varepsilon}(L,n,p,q)|h|^{2\gamma}[D \psi]^2_{B^{\gamma}_{2q-p, \infty}(B_R)} \biggl( \displaystyle\int_{B_{t'}}(1+|D u_j|)^{2q-p}dx \biggr)^{\frac{2q-p-2}{2q-p}} \notag\\
\leq & \varepsilon \displaystyle\int_{\Omega} \eta^2 |\tau_hDu_j|^2 (\mu^2+|Du_j(x+h)|^2+|Du_j(x)|^2)^{\frac{p-2}{2}}dx \notag\\
&+ C_{\varepsilon}(L,n,p,q)|h|^{2\gamma}[D \psi]^{2q-p}_{B^{\gamma}_{2q-p, \infty}(B_R)}+  C_{\varepsilon}(L,n,p,q)|h|^{2\gamma} \displaystyle\int_{B_{t'}}(1+|D u_j|)^{2q-p}dx.  \notag
\end{align}
Therefore, from \eqref{dis}, we infer
\begin{align}
|I_2| \leq & \varepsilon \displaystyle\int_{\Omega} \eta^2 |\tau_hDu_j|^2 (\mu^2+|Du_j(x+h)|^2+|Du_j(x)|^2)^{\frac{p-2}{2}}dx \notag\\
&+ C_{\varepsilon}(L,n,p,q)|h|^{2\gamma}[D \psi]^{2q-p}_{B^{\gamma}_{2q-p, \infty}(B_R)} + C_{\varepsilon}(L,n,p,q)|h|^{2\gamma}\biggl( \displaystyle\int_{B_{t'}}(1+|D u_j|)^{\frac{r(2q-p)}{r-2}}dx \biggr)^{\frac{r-2}{r}} . \label{integI2}
\end{align}

Arguing analogously, we get
\begin{align}
|I_{3}|  \leq & 2L \displaystyle\int_{\Omega} |D \eta| \eta |\tau_{h} Du_j| (1 + |Du_j(x+h)|^{2}+|Du_j(x)|^{2})^{\frac{q-2}{2}}|\tau_{h}(u_j-\psi)| dx \notag\\
\leq & \varepsilon \displaystyle\int_{\Omega} \eta^2 |\tau_hDu_j|^2 (\mu^2+|Du_j(x+h)|^2+|Du_j(x)|^2)^{\frac{p-2}{2}}dx \notag\\
&+ \dfrac{C_{\varepsilon}(L)}{(t-s)^2} \displaystyle\int_{B_t} |\tau_h(u_j- \psi)|^2(\mu^2+ |Du_j(x+h)|^2+|Du_j(x)|^2)^{\frac{2q-p-2}{2}}dx \notag\\
\leq & \varepsilon \displaystyle\int_{\Omega} \eta^2 |\tau_hDu_j|^2 (\mu^2+|Du_j(x+h)|^2+|Du_j(x)|^2)^{\frac{p-2}{2}}dx \notag\\
&+\dfrac{C_{\varepsilon}(L)}{(t-s)^2} \biggl( \displaystyle\int_{B_t}|\tau_h \psi|^{2q-p}dx \biggr)^{\frac{2}{2q-p}}  \biggl( \displaystyle\int_{B_{t'}}(1+|D u_j|)^{2q-p}dx \biggr)^{\frac{2q-p-2}{2q-p}} \notag\\ 
&+\dfrac{C_{\varepsilon}(L)}{(t-s)^2} \biggl( \displaystyle\int_{B_t}|\tau_h u_j|^{2q-p}dx \biggr)^{\frac{2}{2q-p}}  \biggl( \displaystyle\int_{B_{t'}}(1+|D u_j|)^{2q-p}dx \biggr)^{\frac{2q-p-2}{2q-p}}. \notag
\end{align}
Using Young's inequality and Lemma \ref{ldiff}, we obtain
\begin{align}
|I_{3}| \leq & \varepsilon \displaystyle\int_{\Omega} \eta^2 |\tau_hDu_j|^2 (\mu^2+|Du_j(x+h)|^2+|Du_j(x)|^2)^{\frac{p-2}{2}}dx \notag\\
&+\dfrac{C_{\varepsilon}(L,n,p,q)}{(t-s)^2}|h|^2 \biggl( \displaystyle\int_{B_R}|D \psi|^{2q-p}dx \biggr)^{\frac{2}{2q-p}}  \biggl( \displaystyle\int_{B_{t'}}(1+|D u_j|)^{2q-p}dx \biggr)^{\frac{2q-p-2}{2q-p}} \notag\\ 
&+\dfrac{C_{\varepsilon}(L,n,p,q)}{(t-s)^2} |h|^2 \displaystyle\int_{B_{t'}}(1+|D u_j|)^{2q-p}dx \notag\\
\leq & \varepsilon \displaystyle\int_{\Omega} \eta^2 |\tau_hDu_j|^2 (\mu^2+|Du_j(x+h)|^2+|Du_j(x)|^2)^{\frac{p-2}{2}}dx \notag\\
&+\dfrac{C_{\varepsilon}(L,n,p,q)}{(t-s)^2}|h|^2 \displaystyle\int_{B_R}|D \psi|^{2q-p}dx  \notag\\
&+\dfrac{C_{\varepsilon}(L,n,p,q)}{(t-s)^2} |h|^2 \displaystyle\int_{B_{t'}}(1+|D u_j|)^{2q-p}dx .
\end{align}
Recalling the first inequality of \eqref{dis}, we can write
\begin{align}
|I_3| \leq& \varepsilon \displaystyle\int_{\Omega} \eta^2 |\tau_hDu_j|^2 (\mu^2+|Du_j(x+h)|^2+|Du_j(x)|^2)^{\frac{p-2}{2}}dx \notag\\
&+\dfrac{C_{\varepsilon}(L,n,p,q)}{(t-s)^2}|h|^2 \displaystyle\int_{B_R}|D \psi|^{2q-p}dx  \notag\\
&+\dfrac{C_{\varepsilon}(L,n,p,q)}{(t-s)^2} |h|^2 \biggl( \displaystyle\int_{B_{t'}}(1+|D u_j|)^{\frac{r(2q-p)}{r-2}}dx \biggr)^{\frac{r-2}{r}} . \label{integI3}
\end{align}

In order to estimate the integral $I_4$, we use assumption (A4), and Young's and H\"{o}lder's inequalities as follows
\begin{align}
|I_4| \leq & \displaystyle\int_{\Omega} \eta^2 |\tau_h Du_j| |h|^{\alpha} (k(x+h)+k(x))(1+|Du_j(x)|)^{\frac{q-1}{2}}dx \notag\\
\leq & \varepsilon \displaystyle\int_{\Omega} \eta^2 |\tau_hDu_j|^2 (\mu^2+|Du_j(x+h)|^2+|Du_j(x)|^2)^{\frac{p-2}{2}}dx \notag\\
&+ C_{\varepsilon}|h|^{2 \alpha} \displaystyle\int_{B_t} (k(x+h)+k(x))^2(1+|Du_j|)^{2q-p}dx \notag\\
\leq & \varepsilon \displaystyle\int_{\Omega} \eta^2 |\tau_hDu_j|^2 (\mu^2+|Du_j(x+h)|^2+|Du_j(x)|^2)^{\frac{p-2}{2}}dx \notag\\
&+ C_{\varepsilon}|h|^{2 \alpha} \biggl( \displaystyle\int_{B_{R}}k^r dx \biggr)^{\frac{2}{r}} \biggl(  \displaystyle\int_{B_t}(1+|Du_j|)^{\frac{r(2q-p)}{r-2}}dx\biggr)^{\frac{r-2}{r}}. \label{integI4}
\end{align}

We now take care of $I_5$. Similarly as above, exploiting assumption (A4) and H\"{o}lder's inequality, we infer
\begin{align*}
|I_5| 
\leq & \io \eta^2 |\tau_h D\psi| |h|^\alpha \left( k(x+h)+k(x)\right) \left( 1+|Du_j|^2\right)^{\frac{q-1}{2}} dx\\
\leq & |h|^\alpha \left( \ibtt k^r dx \right)^\frac{1}{r} \left( \ibt |\tau_h D\psi|^\frac{r}{r-1}(1+|Du_j|)^\frac{r(q-1)}{r-1} dx \right)^\frac{r-1}{r}\\
\leq & |h|^\alpha \left( \ibR k^r  dx \right)^\frac{1}{r} \left( \ibt |\tau_h D\psi|^{2q-p} dx \right)
^\frac{1}{2q-p}\left( \ibt(1+|Du_j|)^\frac{r(q-1)(2q-p)}{(r-1)(2q-p)-r} dx \right)^\frac{(r-1)(2q-p)-r}{r(2q-p)}.
\end{align*}
Now, we observe
\begin{align}\label{disI5}
\frac{r(q-1)(2q-p)}{(r-1)(2q-p)-r}\leq \frac{r(2q-p)}{r-2} \Leftrightarrow p-2 + r (q-p) \geq 0 ,
\end{align}
which is true by assumption, that is $p \geq 2$, $r > \frac{n}{\alpha}> 2$ and $q>p$.
Hence
\begin{align}\label{integI5}
|I_5| \leq |h|^{\alpha+\gamma} \left( \ibR k^r  dx \right)^\frac{1}{r} [D \psi]_{B^{\gamma}_{2q-p, \infty}(B_R)}\left( \ibt(1+|Du_j|)^\frac{r(2q-p)}{r-2} dx \right)^\frac{(r-2)(q-1)}{r(2q-p)}.
\end{align}

From assumption (A4), hypothesis $|D\eta| < \frac{C}{t-s}$ and H\"{o}lder's inequality, we infer the following estimate for $I_6$.

\begin{align*}
|I_6| \leq &\frac{C}{t-s} |h|^\alpha \ibt |\tau_h \psi| (k(x+h)+k(x))(1+|Du_j|^2)^\frac{q-1}{2} dx\\
&+\frac{C}{t-s} |h|^\alpha \ibt |\tau_h u_j| (k(x+h)+k(x))(1+|Du_j|^2)^\frac{q-1}{2} dx\\
\leq &\frac{C}{t-s} |h|^\alpha \left(\ibR k^r dx \right)^\frac{1}{r} \left(\ibt |\tau_h \psi|^{2q-p} dx\right)^\frac{1}{2q-p}\\
&\cdot \left( \ibt (1+|Du_j|)^\frac{r(q-1)(2q-p)}{(r-1)(2q-p)-r}dx \right)^\frac{(r-1)(2q-p)-r}{r(2q-p)}\\
&+\frac{C}{t-s} |h|^\alpha \left(\ibtt k^r dx \right)^\frac{1}{r} \left( \ibt |\tau_h u_j|^\frac{r}{r-1} (1+|Du_j|)^\frac{r(q-1)}{r-1} dx\right)^\frac{r-1}{r}.
\end{align*}
Using once again H\"{o}lder's inequality, we have
\begin{align*}
|I_6| \leq &\frac{C}{t-s} |h|^{\alpha +1}\left(\ibR k^r dx \right)^\frac{1}{r} \left(\ibtt |D\psi|^{2q-p} dx\right)^\frac{1}{2q-p}\\
&\cdot \left( \ibt (1+|Du_j|)^\frac{r(q-1)(2q-p)}{(r-1)(2q-p)-r}dx \right)^\frac{(r-1)(2q-p)-r}{r(2q-p)}\\
&+\frac{C}{t-s} |h|^\alpha \left(\ibR k^r dx \right)^\frac{1}{r} \left( \ibt |\tau_h u_j|^\frac{rq}{r-1} dx\right)^\frac{r-1}{rq} \left( \ibt (1+|Du_j|)^\frac{rq}{r-1} dx\right)^\frac{(r-1)(q-1)}{rq}.
\end{align*}
Using Lemma \ref{ldiff}, we infer
\begin{align*}
|I_6| \leq &\frac{C}{t-s} |h|^{\alpha +1}\left(\ibR k^r dx \right)^\frac{1}{r} \left(\ibR |D\psi|^{2q-p} dx\right)^\frac{1}{2q-p}\\
&\cdot \left( \ibt (1+|Du_j|)^\frac{r(2q-p)}{r-2}dx \right)^\frac{(q-1)(r-2)}{r(2q-p)}\\
&+\frac{C}{t-s} |h|^{\alpha+1} \left(\ibR k^r dx \right)^\frac{1}{r}  \left( \ibtt (1+|Du_j|)^\frac{rq}{r-1} dx\right)^\frac{r-1}{rq}.
\end{align*}

We remark that 
\begin{equation}\label{dis2}
\dfrac{rq}{r-1} \leq \dfrac{r(2q-p)}{r-2}  \Leftrightarrow p+r(q-p) \geq 0,
\end{equation}
which is true by assumption, that is $p \geq 2$, $r > \frac{n}{\alpha}> 2$ and $q>p$.
Hence
\begin{align}\nonumber
|I_6|  \leq &\frac{C}{t-s}|h|^{\alpha +1}\left(\ibR k^r dx \right)^\frac{1}{r} \left(\ibR |D\psi|^{2q-p} dx\right)^\frac{1}{2q-p}   \\ \nonumber
&\cdot \left( \ibt (1+|Du_j|)^\frac{r(2q-p)}{r-2}dx \right)^\frac{(q-1)(r-2)}{r(2q-p)}\\ \label{integI6}
&+\frac{C}{t-s} |h|^{\alpha+1} \left(\ibR k^r dx \right)^\frac{1}{r} \left( \ibtt (1+|Du_j|)^\frac{r(2q-p)}{r-2} dx\right)^\frac{q(r-2)}{r(2q-p)}. 
\end{align}

Inserting estimates \eqref{I1}, \eqref{integI2}, \eqref{integI3}, \eqref{integI4}, \eqref{integI5} and \eqref{integI6} in \eqref{2:8}, we infer
\begin{align}\label{estimate1.}
 \nu \displaystyle\int_{\Omega}  \eta^{2} & |\tau_{h}Du_j|^{2}(\mu^2 + |Du_j(x+h)|^{2}+|Du_j(x)|^{2})^{\frac{p-2}{2}} dx \notag\\
 \leq & 3\varepsilon \displaystyle\int_{\Omega} \eta^2 |\tau_hDu_j|^2 (\mu^2+|Du_j(x+h)|^2+|Du_j(x)|^2)^{\frac{p-2}{2}}dx \notag\\
 &+ C_{\varepsilon}(L,n,p,q)|h|^{2\gamma}[D \psi]^{2q-p}_{B^{\gamma}_{2q-p, \infty}(B_R)} + C_{\varepsilon}(L,n,p,q)|h|^{2\gamma}\biggl( \displaystyle\int_{B_{t'}}(1+|D u_j|)^{\frac{r(2q-p)}{r-2}}dx \biggr)^{\frac{r-2}{r}} \notag\\
 &+\dfrac{C_{\varepsilon}(L,n,p,q)}{(t-s)^2}|h|^2 \displaystyle\int_{B_R}|D \psi|^{2q-p}dx  \notag\\
&+\dfrac{C_{\varepsilon}(L,n,p,q)}{(t-s)^2} |h|^2 \biggl( \displaystyle\int_{B_{t'}}(1+|D u_j|)^{\frac{r(2q-p)}{r-2}}dx \biggr)^{\frac{r-2}{r}} \notag\\
&+ C_{\varepsilon}|h|^{2 \alpha} \biggl( \displaystyle\int_{B_{R}}k^r dx \biggr)^{\frac{2}{r}} \biggl(  \displaystyle\int_{B_t}(1+|Du_j|)^{\frac{r(2q-p)}{r-2}}dx\biggr)^{\frac{r-2}{r}} \notag\\
& + |h|^{\alpha+\gamma} \left( \ibR k^r  dx \right)^\frac{1}{r} [D \psi]_{B^{\gamma}_{2q-p, \infty}(B_R)}\left( \ibt(1+|Du_j|)^\frac{r(2q-p)}{r-2} dx \right)^\frac{(r-2)(q-1)}{r(2q-p)} \notag\\
&+ \frac{C}{t-s}|h|^{\alpha +1}\left(\ibR k^r dx \right)^\frac{1}{r} \left(\ibR |D\psi|^{2q-p} dx\right)^\frac{1}{2q-p}    \notag \\
&\cdot \left( \ibt (1+|Du_j|)^\frac{r(2q-p)}{r-2}dx \right)^\frac{(q-1)(r-2)}{r(2q-p)} \notag\\ 
&+\frac{C}{t-s} |h|^{\alpha+1} \left(\ibR k^r dx \right)^\frac{1}{r} \left( \ibtt (1+|Du_j|)^\frac{r(2q-p)}{r-2} dx\right)^\frac{q(r-2)}{r(2q-p)}.
\end{align}
We now introduce the following interpolation inequality
\begin{align}\label{interp1}
\Vert Dw \Vert_{\frac{r(2q-p)}{r-2}} \leq \Vert Dw \Vert_p^{\delta} \Vert Dw \Vert_{\frac{np}{n-2 \beta}}^{1- \delta},
\end{align}
where $0< \delta <1$ is defined through the condition
\begin{align}\label{defdelta1}
\dfrac{r-2}{r(2q-p)}= \dfrac{\delta}{p}+ \dfrac{(1-\delta)(n-2 \beta)}{np}
\end{align}
which implies
$$\delta= \dfrac{nr(p-q)-np+\beta r(2q-p)}{\beta r(2q-p)}, \ \ \ 1- \delta= \dfrac{n[r(q-p)+p]}{\beta r(2q-p)}.$$
Hence we get the following inequalities
\begin{align}\label{interpol}
\biggr(\displaystyle\int_{B_{t'}}(1+ |Du_j|)^{\frac{r(2q-p)}{r-2}}dx \biggl)^{\frac{r-2}{r}} \leq & \biggl( \displaystyle\int_{B_{t'}}(1+ |Du_j|)^{p}dx \biggr)^{\frac{\delta(2q-p)}{p}} \notag\\
& \cdot  \biggl( \displaystyle\int_{B_{t'}}(1+ |Du_j|)^{\frac{np}{n-2 \beta}}dx \biggr)^{\frac{(n-2 \beta)[r(q-p)+p]}{\beta pr}},\\
\left( \ibt(1+|Du_j|)^\frac{r(2q-p)}{r-2} dx \right)^\frac{(r-2)(q-1)}{r(2q-p)} 
\leq &\left( \ibt(1+|Du_j|)^p dx \right)^\frac{\delta(q-1)}{p} \notag\\
&\cdot\left( \ibt(1+|Du_j|)^\frac{np}{n-2\beta} dx \right)^\frac{(n-2\beta)(q-1)p'}{p} , \label{int5} \\
\biggr(\displaystyle\int_{B_{t'}}(1+ |Du_j|)^{\frac{r(2q-p)}{r-2}}dx \biggl)^{\frac{q(r-2)}{r(2q-p)}} \leq & \biggl( \displaystyle\int_{B_{t'}}(1+ |Du_j|)^{p}dx \biggr)^{\frac{\delta q}{p}} \notag\\
& \cdot  \biggl( \displaystyle\int_{B_{t'}}(1+ |Du_j|)^{\frac{np}{n-2 \beta}}dx \biggr)^{\frac{(n-2 \beta)q[r(q-p)+p]}{\beta pr(2q-p)}}, \label{int6}
\end{align}
where $p' = \frac{r(q-p)+p}{\beta r (2q-p)}$.

Inserting \eqref{interpol}, \eqref{int5} and \eqref{int6} in \eqref{estimate1.}, and exploiting the bounds 
\begin{gather}\label{ineq.1}
\dfrac{n[r(q-p)+p]}{\beta pr} <1, \quad \frac{n(q-1)[r(q-p)+p]}{\beta r p (2q-p)}<1, \quad \dfrac{nq[r(q-p)+p]}{\beta  pr(2q-p)} <1,
\end{gather}
which hold by assumption \eqref{gap} and for $\beta \in (\frac{n[r(q-p)+p]}{pr}, \alpha)$, from Young's inequality, we infer
\begin{align}
 \nu \displaystyle\int_{\Omega}  \eta^{2} & |\tau_{h}Du_j|^{2}(\mu^2 + |Du_j(x+h)|^{2}+|Du_j(x)|^{2})^{\frac{p-2}{2}} dx \notag\\
 \leq & 3\varepsilon \displaystyle\int_{\Omega} \eta^2 |\tau_hDu_j|^2 (\mu^2+|Du_j(x+h)|^2+|Du_j(x)|^2)^{\frac{p-2}{2}}dx \notag\\
&+ C_{\varepsilon}(L,n,p,q)|h|^{2\gamma}[D \psi]^{2q-p}_{B^{\gamma}_{2q-p, \infty}(B_{R/2})}+ C_{\varepsilon, \theta}(L,n,p,q)|h|^{2 \gamma} \biggl( \displaystyle\int_{B_{R}}(1+ |Du_j|)^{p}dx \biggr)^{\frac{\delta(2q-p)\tilde{p}}{p}} \notag\\
&+ \theta |h|^{2 \gamma}\biggl( \displaystyle\int_{B_{t'}}(1+|D u_j|)^{\frac{np}{n-2 \beta}}dx \biggr)^{\frac{n-2 \beta}{n}}\notag\\
&+\dfrac{C_{\varepsilon}(L,n,p,q)}{(t-s)^2}|h|^2  \displaystyle\int_{B_R}|D \psi|^{2q-p}dx  \notag\\
&+ \theta |h|^2 \biggl( \displaystyle\int_{B_{t'}}(1+|D u_j|)^{\frac{np}{n-2 \beta}}dx \biggr)^{\frac{n-2 \beta}{n}} \notag\\ 
&+\dfrac{C_{\varepsilon, \theta}(L,n,p,q)}{(t-s)^{2\tilde{p}}} |h|^2 \biggl( \displaystyle\int_{B_{R}}(1+|D u_j|)^{p}dx \biggr)^{\frac{\tilde{p}\delta(2q-p)}{p}} \notag\\
&+ C_{\varepsilon, \theta}(n,p,q)|h|^{2 \alpha} \biggr( \displaystyle\int_{B_R} k^r dx \biggl)^{\frac{2 \tilde{p}}{r}} \notag\\
& \cdot \biggl( \displaystyle\int_{B_{R}}(1+|D u_j|)^{p}dx \biggr)^{\frac{\tilde{p}\delta (2q-p)}{p}}\notag\\
&+ \theta |h|^{2 \alpha} \biggl( \displaystyle\int_{B_{t'}}(1+|D u_j|)^{\frac{np}{n-2 \beta}}dx \biggr)^{\frac{n-2 \beta}{n}}\notag\\
&+ C_{\theta}|h|^{\alpha+\gamma} \left( \ibR k^r  dx \right)^\frac{p''}{r} \notag\\ \nonumber
&\cdot[D \psi]^{p''}_{B^{\gamma}_{2q-p, \infty}(B_{R/2})} \notag\\ \nonumber
&\cdot \left( \ibR(1+|Du_j|)^p dx \right)^\frac{\delta(q-1)(2q-p)p''}{p} \notag\\
&+\theta |h|^{\alpha+\gamma}\left( \ibt(1+|Du_j|)^\frac{np}{n-2\beta} dx \right)^\frac{n-2\beta}{n} \notag\\
&+ \frac{C_\theta}{(t-s)^{p''}}|h|^{\alpha +1}\left(\ibR k^r dx \right)^\frac{p''}{r} \left(\ibR |D\psi|^{2q-p} dx\right)^\frac{p''}{2q-p} \notag \\ \nonumber
&\cdot \left( \ibR (1+|Du_j|)^p dx \right)^\frac{\delta(q-1)p''}{p}\\ \nonumber
&+\theta |h|^{\alpha +1}\left( \ibt (1+|Du_j|)^\frac{np}{n-2\beta} dx \right)^\frac{n-2\beta}{n} \notag\\
&+\frac{C_\theta}{(t-s)^{p^*}} |h|^{\alpha+1} \left(\ibR k^r dx \right)^\frac{p^*}{r} \left( \ibR (1+|Du_j|)^{p} dx\right)^{\frac{p^* \delta q}{p}} \notag\\
&+ \theta |h|^{\alpha+1} \left( \ibtt (1+|Du_j|)^\frac{np}{n-2\beta} dx \right)^\frac{n-2\beta}{n}.
\end{align}
for some constant $\theta \in (0,1)$, where we set $\tilde{p}= \frac{\beta p r}{\beta pr -n[r(q-p)+p]}$, $p'' = \frac{\beta r p (2q-p)}{\beta r p (2q-p)-(q-1)n[r(q-p)+p]}$, $p^* = \frac{p}{p  -(1-\delta)q}$.\\
For a better readability we now define
\begin{align*}
A=  & C_{\varepsilon}(L,n,p,q)[D \psi]^{2q-p}_{B^{\gamma}_{2q-p, \infty}(B_{R/2})}+ C_{\varepsilon, \theta}(L,n,p,q)\biggl( \displaystyle\int_{B_{R}}(1+ |Du_j|)^{p}dx \biggr)^{\frac{\delta(2q-p)\tilde{p}}{p}} \notag\\
&+ C_{\varepsilon, \theta}(n,p,q) \biggr( \displaystyle\int_{B_R} k^r dx \biggl)^{\frac{2 \tilde{p}}{r}}  \biggl( \displaystyle\int_{B_{R}}(1+|D u_j|)^{p}dx \biggr)^{\frac{\tilde{p}\delta (2q-p)}{p}} \notag \\
&+  C_{\theta} \left( \ibR k^r  dx \right)^\frac{p''}{r} [D \psi]^{p''}_{B^{\gamma}_{2q-p, \infty}(B_{R/2})}  \left( \ibR(1+|Du_j|)^p dx \right)^\frac{\delta(q-1)(2q-p)p''}{p} \notag\\
B_1=& C_{\varepsilon}(L,n,p,q) \displaystyle\int_{B_R}|D \psi|^{2q-p}dx ,\\
B_2= &C_{\varepsilon, \theta}(L,n,p,q)  \biggl( \displaystyle\int_{B_{R}}(1+|D u_j|)^{p}dx \biggr)^{\frac{\tilde{p}\delta(2q-p)}{p}} ,\\
B_3 = &C_\theta\left(\ibR k^r dx \right)^\frac{p''}{r} \left(\ibR |D\psi|^{2q-p} dx\right)^\frac{p''}{2q-p} \left( \ibR (1+|Du_j|)^p dx \right)^\frac{\delta(q-1)p''}{p},  \\
B_4 = & C_\theta \left(\ibR k^r dx \right)^\frac{p^*}{r} \left( \ibR (1+|Du_j|)^{p} dx\right)^{\frac{p^* \delta q}{p}}, 
\end{align*}
so that we can rewrite the previous estimate as
\begin{align}
 \nu \displaystyle\int_{\Omega}  \eta^{2} & |\tau_{h}Du_j|^{2}(\mu^2 + |Du_j(x+h)|^{2}+|Du_j(x)|^{2})^{\frac{p-2}{2}} dx \notag\\
 \leq &  3\varepsilon \displaystyle\int_{\Omega} \eta^2 |\tau_hDu_j|^2 (\mu^2+|Du_j(x+h)|^2+|Du_j(x)|^2)^{\frac{p-2}{2}}dx \notag\\
 &+  \theta (|h|^{2 \alpha}+|h|^{\alpha + \gamma}+|h|^{\alpha +1} ) \left( \ibt (1+|Du_j|)^\frac{np}{n-2\beta} dx \right)^\frac{n-2\beta}{n} \notag\\
 &+ \theta (|h|^2+|h|^{2 \gamma}+|h|^{\alpha +1} ) \left( \ibtt (1+|Du_j|)^\frac{np}{n-2\beta} dx \right)^\frac{n-2\beta}{n} \notag\\
 &+ (|h|^{2 \gamma}+|h|^{2 \alpha}+|h|^{\alpha + \gamma})A + |h|^2\dfrac{B_1}{(t-s)^2}+ |h|^2\dfrac{B_2}{(t-s)^{2\tilde{p}}} \notag\\
 & + |h|^{\alpha+1}\dfrac{B_3}{(t-s)^{p^{''}}}+ |h|^{\alpha +1}\dfrac{B_4}{(t-s)^{p^*}} .\notag
\end{align}
Choosing $\varepsilon= \frac{\nu}{6}$, we can reabsorb the first integral in the right hand side of the previous estimate by the left hand side, thus getting
\begin{align}
 \displaystyle\int_{\Omega}  \eta^{2} & |\tau_{h}Du_j|^{2}(\mu^2 + |Du_j(x+h)|^{2}+|Du_j(x)|^{2})^{\frac{p-2}{2}} dx \notag\\
 \leq & 3 \theta |h|^{2 \alpha} \left( \ibt (1+|Du_j|)^\frac{np}{n-2\beta} dx \right)^\frac{n-2\beta}{n} + 3\theta |h|^{2 \alpha} \left( \ibtt (1+|Du_j|)^\frac{np}{n-2\beta} dx \right)^\frac{n-2\beta}{n} \notag\\
 &+ |h|^{2 \alpha}A + |h|^2\dfrac{B_1}{(t-s)^2}+ |h|^2\dfrac{B_2}{(t-s)^{\tilde{p}}}  + |h|^{2\alpha}\dfrac{B_3}{(t-s)^{p^{''}}}+ |h|^{2\alpha }\dfrac{B_4}{(t-s)^{p^*}} ,\notag
\end{align}
where we used the fact that $\alpha < \gamma$. Using Lemma \ref{D1} in the left hand side of the previous inequality, recalling that $\eta = 1$ on $B_s$, we get
\begin{align}
 \displaystyle\int_{B_s}  |\tau_h V_p(Du_j)|^2 dx \leq  |h|^{2 \alpha} \biggl\{ & 3 \theta  \left( \ibt (1+|Du_j|)^\frac{np}{n-2\beta} dx \right)^\frac{n-2\beta}{n} + 3\theta  \left( \ibtt (1+|Du_j|)^\frac{np}{n-2\beta} dx \right)^\frac{n-2\beta}{n}  \notag\\
 &+ A + \dfrac{B_1}{(t-s)^2}+ \dfrac{B_2}{(t-s)^{\tilde{p}}}  + \dfrac{B_3}{(t-s)^{p^{''}}}+ \dfrac{B_4}{(t-s)^{p^*}} \biggr\}. \label{tauVp}  
\end{align}
Lemma \ref{Besovlem} and inequality \eqref{Vp} imply
\begin{align}
\biggl( \displaystyle\int_{B_s}  |Du_j|^{\frac{np}{n- 2 \beta }} dx \biggr)^{\frac{n- 2 \beta}{n}} \leq & 3 \theta  \left( \ibt (1+|Du_j|)^\frac{np}{n-2\beta} dx \right)^\frac{n-2\beta}{n} + 3\theta  \left( \ibtt (1+|Du_j|)^\frac{np}{n-2\beta} dx \right)^\frac{n-2\beta}{n} \notag\\
 &+ A + \dfrac{B_1}{(t-s)^2}+ \dfrac{B_2}{(t-s)^{2\tilde{p}}}  + \dfrac{B_3}{(t-s)^{p^{''}}}+ \dfrac{B_4}{(t-s)^{p^*}} ,\label{gradest}
\end{align}
for all $\beta \in (0, \alpha )$.\\
Setting
\begin{gather}
\Phi (r) = \biggl( \displaystyle\int_{B_r}  |Du_j|^{\frac{np}{n- 2 \beta }} dx \biggr)^{\frac{n- 2 \beta}{n}}, \notag
\end{gather}
we can write inequality \eqref{gradest} as
\begin{gather}
\Phi (s) \leq 3 \theta \Phi(t) +3 \theta \Phi(t')  + A + \dfrac{B_1}{(t-s)^2}+ \dfrac{B_2}{(t-s)^{2\tilde{p}}}  + \dfrac{B_3}{(t-s)^{p^{''}}}+ \dfrac{B_4}{(t-s)^{p^*}}.
\end{gather}
By virtue of Lemma \ref{lm2}, choosing $0 < \theta < 1/3$, we obtain
\begin{gather}
\Phi (\varrho) \leq c \biggl( 3 \theta \Phi(t')  + A + \dfrac{B_1}{R^2}+ \dfrac{B_2}{R^{2\tilde{p}}}  + \dfrac{B_3}{R^{p^{''}}}+ \dfrac{B_4}{R^{p^*}} \biggr),
\end{gather}
for some constant $c := c(n,p,q,r, \beta, \theta)$. Then, applying Lemma \ref{lm2} again, we get
\begin{gather}
\Phi \biggl( \dfrac{R}{4}  \biggr) \leq \tilde{c} \biggl(  A + \dfrac{B_1}{R^2}+ \dfrac{B_2}{R^{2\tilde{p}}}  + \dfrac{B_3}{R^{p^{''}}}+ \dfrac{B_4}{R^{p^*}} \biggr),
\end{gather}
with $\tilde{c} := \tilde{c}(n,p,q,r ,\beta ,\theta)$.
\\Now, recalling the definition of $\Phi$, we obtain
\begin{gather}\label{apriorigrad}
\biggl( \displaystyle\int_{B_{R/4}}  |Du_j|^{\frac{np}{n- 2 \beta }} dx \biggr)^{\frac{n- 2 \beta}{n}} \leq  \tilde{c}   \biggl\{  \ibR (1+ |Du_j|^p)dx+ \Vert D \psi \Vert_{B^{\gamma}_{2q-p, \infty}(B_{R/2})} \biggl\}^{\kappa},
\end{gather}
thus, using Lemma \ref{Besovlem}, from inequalities \eqref{apriorigrad} and \eqref{tauVp}, we deduce the a priori estimate
\begin{gather}
\displaystyle\int_{B_{R/4}} |\tau_h V_p(Du_j)|^2 dx \leq C |h|^{2\alpha} \biggl\{  \ibR (1+ |Du_j|^p)dx+ \Vert D \psi \Vert_{B^{\gamma}_{2q-p, \infty}(B_{R/2})} \biggl\}^{\kappa},
\end{gather}
for some $\beta < \alpha$, where $C := C(R,n,p,q,r, \beta)$ and $\kappa:= \kappa (n,p,q,r, \beta)$. \endproof

\subsection{Passage to the limit}\label{passagetothelimit}
\label{secproof}
Let $u \in \mathcal{K}_{\psi}(\Omega)$ be a solution to \eqref{1}, and let $F_j$ be defined as in Lemma \ref{apprlem1}. From Theorem \ref{thmdefF}, there exists $c_1 > 0$ such that
\begin{equation}\label{F_j}
|\xi|^p \leq c_1 (1+ F_j(x,\xi)), \quad \forall j \in \mathbb{N}.
\end{equation}
Fixed $B_R \Subset \Omega$, let $u_j$ be the solution of the problem
$$\min \biggl\{ \displaystyle\int_{B_R} F_j(x,Dw) dx :w \geq \psi \ \text{a.e. in} \ B_R, \ w \in u + W^{1,p}_0(B_R)  \biggr\}.$$
From \eqref{F_j}, the minimality of $u_j$ implies
\begin{align}\label{intF_j}
\displaystyle\int_{B_R}|Du_j|^p dx \leq & c_1 \displaystyle\int_{B_R} (1+ F_j(x,Du_j))dx \notag \\
\leq & c_1 \displaystyle\int_{B_R} (1+ F_j(x,Du))dx \notag \\
\leq & c_1 \displaystyle\int_{B_R} (1+ F(x,Du))dx, 
\end{align}
where in the last inequality we used Lemma \ref{apprlem1} $(ii)$. Thus, up to subsequences, 
\begin{equation}\label{weakconv}
u_j \rightharpoonup \tilde{u} \ \text{in} \ u+W^{1,p}_{0}(B_R)
\end{equation}
and
\begin{equation}\label{strongconv}
u_j \rightarrow \tilde{u} \ \text{in} \ L^p(B_R).
\end{equation}
\\For any $j$, $F_j$ satisfies the assumptions of Theorem \ref{approximation}. Combining \eqref{apriorigrad} and \eqref{intF_j} we get
\begin{gather}\label{boundforgrad}
\Vert Du_j\Vert_{L^{\frac{np}{n-2 \beta}}(B_{R/4})} \leq  \tilde{c}   \biggl\{  \ibR (1+ F(x,Du))dx+ \Vert D \psi \Vert_{B^{\gamma}_{2q-p, \infty}(B_R)} \biggl\}^{\tilde{\kappa}},
\end{gather}
thus, by \eqref{weakconv}, \eqref{boundforgrad} and weak lower semicontinuity, we infer
\begin{align}\label{bgrad}
\Vert D\tilde{u} \Vert_{L^{\frac{np}{n-2 \beta}}(B_{R/4})} \leq & \displaystyle\liminf_{j \rightarrow \infty} \Vert Du_j\Vert_{L^{\frac{np}{n-2 \beta}}(B_{R/4})} \notag\\
\leq &\tilde{c}   \biggl\{  \ibR (1+ F(x,Du))dx+ \Vert D \psi \Vert_{B^{\gamma}_{2q-p, \infty}(B_R)} \biggl\}^{\tilde{\kappa}}.
\end{align}
By weak lower semicontinuity of the functional
$v \mapsto \int_{B_R}F_j(x,Dv(x))dx$, \eqref{weakconv}, Lemma \ref{apprlem1} (ii) and minimality of the $u_j$'s, we have
\begin{align}
\ibR F(x,D\tilde{u})dx & \leq \displaystyle\liminf_{j \rightarrow \infty} \ibR  F(x, D u_j) dx \notag \\
& \leq \displaystyle\liminf_{j \rightarrow \infty} \ibR F_j(x,Du_j) dx \notag\\
& \leq \ibR F(x,Du) dx.
\end{align}
Moreover, by the weak convergence \eqref{weakconv}, the limit function $\tilde{u}$ still belongs to $\mathcal{K}_{\psi}(B_R)$, since this set is convex and closed. Thus, we can conclude that
\begin{equation}
    \tilde{u}= u \quad \text{a.e. in} \ B_R
\end{equation}
by strict convexity of $F$, and, recalling estimate \eqref{bgrad},
\begin{align}
\Vert Du \Vert_{L^{\frac{np}{n-2 \beta}}(B_{R/4})} 
\leq \tilde{c}   \biggl\{  \ibR (1+ F(x,Du))dx+ \Vert D \psi \Vert_{B^{\gamma}_{2q-p, \infty}(B_R)} \biggl\}^{\tilde{\kappa}}.
\end{align}
Finally, we can repeat the proof of Theorem \ref{approximation} obtaining $V_p(Du) \in B^{\alpha}_{2, \infty,loc}(\Omega)$.

\section{Proof of Theorem \ref{mainthm2}}\label{secthm2}
This section is devoted to the proof of Theorem \ref{mainthm2}. We here focus only on the derivation of the a priori estimate. Indeed, the limit procedure is achieved using the same arguments presented in Sections \ref{secappr} (cnfr. Remark \ref{remappr}) and \ref{secproof}.\\

We a priori assume that the map $\A$ satisfies appropriate growth conditions so that the integral
$$\displaystyle\int_{\Omega'}(1+ |Du_j|)^{\frac{np}{n-2\lambda}}dx$$
is finite, for every $\Omega' \Subset \Omega$, where we denote $\lambda = \min \{ \alpha, \gamma \}$.

Arguing analogously as in the proof of Theorem \ref{approximation}, we define the integrals $I_1$--$I_6$ according to \eqref{2:7} and we are able to derive estimates \eqref{2:8} and \eqref{I1}. We need to treat differently the integrals $I_2$ -- $I_6$ in which the new assumptions (A5) on the gradient of the obstacle and on the map $x \mapsto \A(x,\xi)$ come into the play.\\
Similarly as we did for \eqref{dis} but using this time \eqref{gap2}, we get
\begin{equation}\label{dis1.}
2q-p \leq \dfrac{r(2q-p)}{r-2} \leq \dfrac{np}{n-2 \lambda}.
\end{equation} 
Consider the integral $I_2$, then according to $L^p$ embeddings and Young's inequality,
\begin{align}\label{integrI2}
|I_{2}|\leq & \varepsilon \displaystyle\int_{\Omega} \eta^2 |\tau_hDu_j|^2 (\mu^2+|Du_j(x+h)|^2+|Du_j(x)|^2)^{\frac{p-2}{2}}dx \notag\\
&+ C_{\varepsilon}(L) \biggl( \displaystyle\int_{B_t}|\tau_h D \psi|^{2q-p}dx \biggr)^{\frac{2}{2q-p}}  \biggl( \displaystyle\int_{B_{t'}}(1+|D u_j|)^{2q-p}dx \biggr)^{\frac{2q-p-2}{2q-p}}\notag\\
\leq & \varepsilon \displaystyle\int_{\Omega} \eta^2 |\tau_hDu_j|^2 (\mu^2+|Du_j(x+h)|^2+|Du_j(x)|^2)^{\frac{p-2}{2}}dx \notag\\
&+ C_{\varepsilon}(L) \biggl( \displaystyle\int_{B_{R/2}}|\tau_h D \psi|^{2q-p}dx \biggr)^{\frac{2}{2q-p}}  \biggl( \displaystyle\int_{B_{t'}}(1+|D u_j|)^{\frac{r(2q-p)}{r-2}}dx \biggr)^{\frac{(r-2)(2q-p-2)}{r(2q-p)}} ,
\end{align}
where in the last inequality we used \eqref{dis1.}.

In order to take care of $I_3$, we are able to perform the same computations which led us to \eqref{integI3}, that is
\begin{align}
|I_3| \leq& \varepsilon \displaystyle\int_{\Omega} \eta^2 |\tau_hDu_j|^2 (\mu^2+|Du_j(x+h)|^2+|Du_j(x)|^2)^{\frac{p-2}{2}}dx \notag\\
&+\dfrac{C_{\varepsilon}(L,n,p,q)}{(t-s)^2}|h|^2 \displaystyle\int_{B_R}|D \psi|^{2q-p}dx  \notag\\
&+\dfrac{C_{\varepsilon}(L,n,p,q)}{(t-s)^2} |h|^2 \biggl( \displaystyle\int_{B_{t'}}(1+|D u_j|)^{\frac{r(2q-p)}{r-2}}dx \biggr)^{\frac{r-2}{r}} . \label{integI3'}
\end{align}

Now, we estimate the integral $I_4$. Assumption (A5), Young's and H\"{o}lder's inequalities yield that
\begin{align}
|I_4| \leq & \displaystyle\int_{\Omega} \eta^2 |\tau_h Du_j| |h|^{\alpha} (g_k(x+h)+g_k(x))(1+|Du_j(x)|)^{\frac{q-1}{2}}dx \notag\\
\leq & \varepsilon \displaystyle\int_{\Omega} \eta^2 |\tau_hDu|^2 (\mu^2+|Du_j(x+h)|^2+|Du_j(x)|^2)^{\frac{p-2}{2}}dx \notag\\
&+ C_{\varepsilon}|h|^{2 \alpha} \displaystyle\int_{B_t} (g_k(x+h)+g_k(x))^2(1+|Du_j|)^{2q-p}dx \notag\\
\leq & \varepsilon \displaystyle\int_{\Omega} \eta^2 |\tau_hDu|^2 (\mu^2+|Du_j(x+h)|^2+|Du_j(x)|^2)^{\frac{p-2}{2}}dx \notag\\
&+ C_{\varepsilon}|h|^{2 \alpha} \biggl( \displaystyle\int_{B_{t}}(g_k(x+h)+g_k(x))^r dx \biggr)^{\frac{2}{r}} \biggl(  \displaystyle\int_{B_t}(1+|Du_j|)^{\frac{r(2q-p)}{r-2}}dx\biggr)^{\frac{r-2}{r}} .
 \label{i4.}
\end{align}
Exploiting assumption (A5) and H\"{o}lder's inequality, we infer the following estimate for the integral $I_5$
\begin{align}\label{integrI5}
|I_5| 
\leq & \io \eta^2 |\tau_h D\psi| |h|^\alpha \left( g_k(x+h)+g_k(x)\right) \left( 1+|Du_j|^2\right)^{\frac{q-1}{2}} dx\notag\\
\leq & |h|^\alpha \left( \ibt (g_k(x+h)+g_k(x))^r  dx \right)^\frac{1}{r} \left( \ibt |\tau_h D\psi|^{2q-p} dx \right)
^\frac{1}{2q-p} \notag\\
& \cdot \left( \ibt(1+|Du_j|)^\frac{r(q-1)(2q-p)}{(r-1)(2q-p)-r} dx \right)^\frac{(r-1)(2q-p)-r}{r(2q-p)} \notag\\
\leq & |h|^\alpha \left( \displaystyle\int_{B_{R/2}}(g_k(x+h)+g_k(x))^r  dx \right)^\frac{1}{r} \left( \displaystyle\int_{B_{R/2}} |\tau_h D\psi|^{2q-p} dx \right)
^\frac{1}{2q-p} \notag\\
& \cdot \left( \ibt(1+|Du_j|)^\frac{r(2q-p)}{r-2} dx \right)^\frac{(r-2)(q-1)}{r(2q-p)},
\end{align}
where in the last inequality we used \eqref{disI5}.

Similarly as above, from assumption (A5), \eqref{disI5}, hypothesis $|D\eta| < \frac{C}{t-s}$ and H\"{o}lder's inequality, we can estimate the integral $I_6$ as follows
\begin{align}\label{integr.I6}
|I_6| \leq &\frac{C}{t-s} |h|^\alpha \ibt |\tau_h \psi| (g_k(x+h)+g_k(x))(1+|Du_j|^2)^\frac{q-1}{2} dx\notag\\
&+\frac{C}{t-s} |h|^\alpha \ibt |\tau_h u_j| (g_k(x+h)+g_k(x))(1+|Du_j|^2)^\frac{q-1}{2} dx\notag\\
\leq &\frac{C}{t-s} |h|^\alpha \left(\ibt (g_k(x+h)+g_k(x))^r dx \right)^\frac{1}{r} \left(\ibt |\tau_h \psi|^{2q-p} dx\right)^\frac{1}{2q-p}\notag\\
&\cdot \left( \ibt (1+|Du_j|)^\frac{r(q-1)(2q-p)}{(r-1)(2q-p)-r}dx \right)^\frac{(r-1)(2q-p)-r}{r(2q-p)}\notag\\
&+\frac{C}{t-s} |h|^\alpha \left(\ibt (g_k(x+h)+g_k(x))^r dx \right)^\frac{1}{r} \left( \ibt |\tau_h u_j|^\frac{rq}{r-1} dx\right)^\frac{r-1}{rq} \notag\\
& \cdot\left( \ibt (1+|Du_j|)^\frac{rq}{r-1} dx\right)^\frac{(r-1)(q-1)}{rq} \notag\\
\leq &\frac{C}{t-s} |h|^{\alpha+1} \left(\displaystyle\int_{B_{R/2}} (g_k(x+h)+g_k(x))^r dx \right)^\frac{1}{r} \left(\ibR |D \psi|^{2q-p} dx\right)^\frac{1}{2q-p}\notag\\
&\cdot \left( \ibt (1+|Du_j|)^\frac{r(2q-p)}{r-2}dx \right)^\frac{(r-2)(q-1)}{r(2q-p)}\notag\\
&+\frac{C}{t-s} |h|^{\alpha+1} \left(\displaystyle\int_{B_{R/2}} (g_k(x+h)+g_k(x))^r dx \right)^\frac{1}{r}  \left( \ibtt (1+|Du_j|)^\frac{rq}{r-1} dx\right)^\frac{r-1}{r}.
\end{align}
Inserting estimates \eqref{I1}, \eqref{integrI2}, \eqref{integI3'}, \eqref{i4.}, \eqref{integrI5} and \eqref{integr.I6} in \eqref{2:8}, we infer
\begin{align}\label{formula}
\nu \displaystyle\int_{\Omega} & \eta^2 |\tau_hDu|^2 (\mu^2+|Du_j(x+h)|^2+|Du_j(x)|^2)^{\frac{p-2}{2}}dx \notag\\
  \leq & 3 \varepsilon \displaystyle\int_{\Omega} \eta^2 |\tau_hDu|^2 (\mu^2+|Du_j(x+h)|^2+|Du_j(x)|^2)^{\frac{p-2}{2}}dx \notag\\
  &+ C_{\varepsilon}(L) \biggl( \displaystyle\int_{B_{R/2}}|\tau_h D \psi|^{2q-p}dx \biggr)^{\frac{2}{2q-p}}  \biggl( \displaystyle\int_{B_{t'}}(1+|D u_j|)^{\frac{r(2q-p)}{r-2}}dx \biggr)^{\frac{(r-2)(2q-p-2)}{r(2q-p)}} \notag \\
  &+\dfrac{C_{\varepsilon}(L,n,p,q)}{(t-s)^2}|h|^2 \displaystyle\int_{B_R}|D \psi|^{2q-p}dx  \notag\\
&+\dfrac{C_{\varepsilon}(L,n,p,q)}{(t-s)^2} |h|^2 \biggl( \displaystyle\int_{B_{t'}}(1+|D u_j|)^{\frac{r(2q-p)}{r-2}}dx \biggr)^{\frac{r-2}{r}}\notag \\
&+ C_{\varepsilon}|h|^{2 \alpha} \biggl( \displaystyle\int_{B_{t}}(g_k(x+h)+g_k(x))^r dx \biggr)^{\frac{2}{r}} \biggl(  \displaystyle\int_{B_t}(1+|Du_j|)^{\frac{r(2q-p)}{r-2}}dx\biggr)^{\frac{r-2}{r}} \notag\\
& +|h|^\alpha \left( \displaystyle\int_{B_{R/2}}(g_k(x+h)+g_k(x))^r  dx \right)^\frac{1}{r} \left( \displaystyle\int_{B_{R/2}} |\tau_h D\psi|^{2q-p} dx \right)
^\frac{1}{2q-p} \notag\\
& \cdot \left( \ibt(1+|Du_j|)^\frac{r(2q-p)}{r-2} dx \right)^\frac{(r-2)(q-1)}{r(2q-p)} \notag\\
 &+\frac{C}{t-s} |h|^{\alpha+1} \left(\displaystyle\int_{B_{R/2}} (g_k(x+h)+g_k(x))^r dx \right)^\frac{1}{r} \left(\ibR |D \psi|^{2q-p} dx\right)^\frac{1}{2q-p} \notag\\
&\cdot \left( \ibt (1+|Du_j|)^\frac{r(2q-p)}{r-2}dx \right)^\frac{(r-2)(q-1)}{r(2q-p)}\notag\\
&+\frac{C}{t-s} |h|^{\alpha+1} \left(\displaystyle\int_{B_{R/2}} (g_k(x+h)+g_k(x))^r dx \right)^\frac{1}{r}  \left( \ibtt (1+|Du_j|)^\frac{rq}{r-1} dx\right)^\frac{r-1}{r} .
\end{align}

Replacing $\beta$ with $\lambda$ in \eqref{interp1}, we get the following interpolation inequality 
\begin{align}\label{interp1.}
\Vert Dw \Vert_{\frac{r(2q-p)}{r-2}} \leq \Vert Dw \Vert_p^{\delta} \Vert Dw \Vert_{\frac{np}{n-2 \lambda}}^{1- \delta},
\end{align}
where $0< \delta <1$ is defined through the condition
\begin{align}\label{defdelta1.}
\dfrac{r-2}{r(2q-p)}= \dfrac{\delta}{p}+ \dfrac{(1-\delta)(n-2 \lambda)}{np}.
\end{align}

Hence, using the interpolation inequality \eqref{interp1.}, from estimate \eqref{formula}, we infer
\begin{align}
\nu \displaystyle\int_{\Omega} & \eta^2 |\tau_hDu|^2 (\mu^2+|Du_j(x+h)|^2+|Du_j(x)|^2)^{\frac{p-2}{2}}dx \notag\\
  \leq & 3 \varepsilon \displaystyle\int_{\Omega} \eta^2 |\tau_hDu|^2 (\mu^2+|Du_j(x+h)|^2+|Du_j(x)|^2)^{\frac{p-2}{2}}dx \notag\\
&+ C_{\varepsilon}(L,p,q) \biggl( \displaystyle\int_{B_{R/2}}|\tau_h D \psi|^{2q-p}dx \biggr)^{\frac{2}{2q-p}}  \biggl( \displaystyle\int_{B_{R}}(1+|D u_j|)^{p}dx \biggr)^{\frac{\delta (2q-p-2)}{p}}  \notag\\
& \cdot \biggl( \displaystyle\int_{B_{t'}}(1+|D u_j|)^{\frac{np}{n-2 \lambda}}dx \biggr)^{\frac{(1-\delta )(n-2 \lambda)(2q-p-2)}{np}} \notag\\
&+\dfrac{C_{\varepsilon}(L,n,p,q)}{(t-s)^2}|h|^2 \displaystyle\int_{B_R}|D \psi|^{2q-p}dx  \notag\\
&+\dfrac{C_{\varepsilon}(L,n,p,q)}{(t-s)^2} |h|^2 \biggl( \displaystyle\int_{B_{R}}(1+|D u_j|)^{p}dx \biggr)^{\frac{\delta(2q-p)}{p}} \notag\\
& \cdot \biggl( \displaystyle\int_{B_{t'}}(1+|D u_j|)^{\frac{np}{n-2 \lambda}}dx \biggr)^{\frac{(1-\delta)(n-2 \lambda)(2q-p)}{np}} \notag\\
&+ C_{\varepsilon}|h|^{2 \alpha} \biggl( \displaystyle\int_{B_{R/2}}(g_k(x+h)+g_k(x))^r dx \biggr)^{\frac{2}{r}} \biggl( \displaystyle\int_{B_{R}}(1+|D u_j|)^{p}dx \biggr)^{\frac{\delta(2q-p)}{p}} \notag\\
& \cdot \biggl( \displaystyle\int_{B_{t}}(1+|D u_j|)^{\frac{np}{n-2 \lambda}}dx \biggr)^{\frac{(1-\delta)(n-2 \lambda)(2q-p)}{np}} \notag\\
&+ |h|^\alpha \left( \displaystyle\int_{B_{R/2}} (g_k(x+h)+g_k(x))^r  dx \right)^\frac{1}{r} \left( \displaystyle\int_{B_{R/2}} |\tau_h D\psi|^{2q-p} dx \right)
^\frac{1}{2q-p} \notag\\
& \cdot  \biggl( \displaystyle\int_{B_{R}}(1+|D u_j|)^{p}dx \biggr)^{\frac{\delta(q-1)}{p}}  \biggl( \displaystyle\int_{B_{t}}(1+|D u_j|)^{\frac{np}{n-2 \lambda}}dx \biggr)^{\frac{(1-\delta)(n-2 \lambda)(q-1)}{np}}\notag\\
&+\frac{C}{t-s} |h|^{\alpha+1} \left(\displaystyle\int_{B_{R/2}} (g_k(x+h)+g_k(x))^r dx \right)^\frac{1}{r} \left(\ibR |D \psi|^{2q-p} dx\right)^\frac{1}{2q-p} \notag\\
&\cdot \biggl( \displaystyle\int_{B_{R}}(1+|D u_j|)^{p}dx \biggr)^{\frac{\delta(q-1)}{p}}  \biggl( \displaystyle\int_{B_{t}}(1+|D u_j|)^{\frac{np}{n-2 \lambda}}dx \biggr)^{\frac{(1-\delta)(n-2 \lambda)(q-1)}{np}}\notag\\
&+\frac{C}{t-s} |h|^{\alpha+1} \left(\displaystyle\int_{B_{R/2}} (g_k(x+h)+g_k(x))^r dx \right)^\frac{1}{r}  \biggl( \displaystyle\int_{B_{R}}(1+|D u_j|)^{p}dx \biggr)^{\frac{\delta q}{p}} \notag\\
& \cdot  \biggl( \displaystyle\int_{B_{t'}}(1+|D u_j|)^{\frac{np}{n-2 \lambda}}dx \biggr)^{\frac{(1-\delta)(n-2 \lambda)q}{np}} .
\end{align}
Choosing $\varepsilon=\frac{\nu}{6}$ yields
\begin{align}
\nu \displaystyle\int_{\Omega} & \eta^2 |\tau_hDu|^2 (\mu^2+|Du_j(x+h)|^2+|Du_j(x)|^2)^{\frac{p-2}{2}}dx \notag\\
  \leq & C \biggl( \displaystyle\int_{B_{R/2}}|\tau_h D \psi|^{2q-p}dx \biggr)^{\frac{2}{2q-p}}  \biggl( \displaystyle\int_{B_{R}}(1+|D u_j|)^{p}dx \biggr)^{\frac{\delta (2q-p-2)}{p}}  \notag\\
& \cdot \biggl( \displaystyle\int_{B_{t'}}(1+|D u_j|)^{\frac{np}{n-2 \lambda}}dx \biggr)^{\frac{(1-\delta )(n-2 \lambda)(2q-p-2)}{np}} \notag\\
&+\dfrac{C}{(t-s)^2}|h|^2 \displaystyle\int_{B_R}|D \psi|^{2q-p}dx  \notag\\
&+\dfrac{C}{(t-s)^2} |h|^2 \biggl( \displaystyle\int_{B_{R}}(1+|D u_j|)^{p}dx \biggr)^{\frac{\delta(2q-p)}{p}} \notag\\
& \cdot \biggl( \displaystyle\int_{B_{t'}}(1+|D u_j|)^{\frac{np}{n-2 \lambda}}dx \biggr)^{\frac{(1-\delta)(n-2 \lambda)(2q-p)}{np}} \notag\\
&+ C|h|^{2 \alpha} \biggl( \displaystyle\int_{B_{R/2}}(g_k(x+h)+g_k(x))^r dx \biggr)^{\frac{2}{r}} \biggl( \displaystyle\int_{B_{R}}(1+|D u_j|)^{p}dx \biggr)^{\frac{\delta(2q-p)}{p}} \notag\\
& \cdot \biggl( \displaystyle\int_{B_{t}}(1+|D u_j|)^{\frac{np}{n-2 \lambda}}dx \biggr)^{\frac{(1-\delta)(n-2 \lambda)(2q-p)}{np}} \notag\\
&+ |h|^\alpha \left( \displaystyle\int_{B_{R/2}} (g_k(x+h)+g_k(x))^r  dx \right)^\frac{1}{r} \left( \displaystyle\int_{B_{R/2}} |\tau_h D\psi|^{2q-p} dx \right)
^\frac{1}{2q-p} \notag\\
& \cdot  \biggl( \displaystyle\int_{B_{R}}(1+|D u_j|)^{p}dx \biggr)^{\frac{\delta(q-1)}{p}}  \biggl( \displaystyle\int_{B_{t}}(1+|D u_j|)^{\frac{np}{n-2 \lambda}}dx \biggr)^{\frac{(1-\delta)(n-2 \lambda)(q-1)}{np}}\notag\\
&+\frac{C}{t-s} |h|^{\alpha+1} \left(\displaystyle\int_{B_{R/2}} (g_k(x+h)+g_k(x))^r dx \right)^\frac{1}{r} \left(\ibR |D \psi|^{2q-p} dx\right)^\frac{1}{2q-p} \notag\\
&\cdot \biggl( \displaystyle\int_{B_{R}}(1+|D u_j|)^{p}dx \biggr)^{\frac{\delta(q-1)}{p}}  \biggl( \displaystyle\int_{B_{t}}(1+|D u_j|)^{\frac{np}{n-2 \lambda}}dx \biggr)^{\frac{(1-\delta)(n-2 \lambda)(q-1)}{np}}\notag\\
&+\frac{C}{t-s} |h|^{\alpha+1} \left(\displaystyle\int_{B_{R/2}} (g_k(x+h)+g_k(x))^r dx \right)^\frac{1}{r}  \biggl( \displaystyle\int_{B_{R}}(1+|D u_j|)^{p}dx \biggr)^{\frac{\delta q}{p}} \notag\\
& \cdot  \biggl( \displaystyle\int_{B_{t'}}(1+|D u_j|)^{\frac{np}{n-2 \lambda}}dx \biggr)^{\frac{(1-\delta)(n-2 \lambda)q}{np}} .
\end{align}
for a positive constant $C := C(L,n,p,q)$.

Using Lemma \ref{D1} in the left hand side of previous estimate, recalling that $\eta =1 $ on $B_s$ and dividing both sides by $|h|^{2 \lambda}$, we get
\begin{align}
\nu \displaystyle\int_{B_s} & \dfrac{|\tau_hV_p(Du_j)|^2}{|h|^{2 \lambda}} dx \notag\\
  \leq 
& C \biggl( \displaystyle\int_{B_{R/2}}\dfrac{|\tau_h D \psi|^{2q-p}}{|h|^{\lambda(2q-p)}}dx \biggr)^{\frac{2}{2q-p}}  \biggl( \displaystyle\int_{B_{R}}(1+|D u_j|)^{p}dx \biggr)^{\frac{\delta (2q-p-2)}{p}} \notag\\
& \cdot  \biggl( \displaystyle\int_{B_{t'}}(1+|D u_j|)^{\frac{np}{n-2 \lambda}}dx \biggr)^{\frac{(1-\delta )(n-2 \lambda)(2q-p-2)}{np}} \notag\\
&+\dfrac{C}{(t-s)^2}|h|^{2(1-\lambda)}  \biggl\{\displaystyle\int_{B_R}|D \psi|^{2q-p}dx  + \biggl( \displaystyle\int_{B_{R}}(1+|D u_j|)^{p}dx \biggr)^{\frac{\delta(2q-p)}{p}} \notag\\
& \cdot \biggl( \displaystyle\int_{B_{t'}}(1+|D u_j|)^{\frac{np}{n-2 \lambda}}dx \biggr)^{\frac{(1-\delta)(n-2 \lambda)(2q-p)}{np}} \biggr\} \notag\\
&+ C|h|^{2 (\alpha-\lambda)} \biggl( \displaystyle\int_{B_{R/2}}(g_k(x+h)+g_k(x))^r dx \biggr)^{\frac{2}{r}} \biggl( \displaystyle\int_{B_{R}}(1+|D u_j|)^{p}dx \biggr)^{\frac{\delta (2q-p)}{p}} \notag\\
& \cdot\biggl( \displaystyle\int_{B_{t}}(1+|D u_j|)^{\frac{np}{n-2 \lambda}}dx \biggr)^{\frac{(1-\delta)(n-2 \lambda)(2q-p)}{np}} \notag\\
&+ C|h|^{\alpha-\lambda} \left( \displaystyle\int_{B_{R/2}} (g_k(x+h)+g_k(x))^r  dx \right)^\frac{1}{r} \left( \displaystyle\int_{B_{R/2}} \dfrac{|\tau_h D\psi|^{2q-p}}{|h|^{\lambda (2q-p)}} dx \right)^\frac{1}{2q-p} \notag\\
&\cdot \biggl( \displaystyle\int_{B_{R}}(1+|D u_j|)^{p}dx \biggr)^{\frac{\delta (q-1)}{p}} \biggl( \displaystyle\int_{B_{t}}(1+|D u_j|)^{\frac{np}{n-2 \lambda}}dx \biggr)^{\frac{(1-\delta)(n-2 \lambda)(q-1)}{np}}\notag\\
&+\frac{C}{t-s} |h|^{\alpha+1-2 \lambda} \left(\displaystyle\int_{B_{R/2}} (g_k(x+h)+g_k(x))^r dx \right)^\frac{1}{r} \notag\\
& \cdot \biggl\{ \biggl( \displaystyle\int_{B_{R}}(1+|D u_j|)^{p}dx \biggr)^{\frac{\delta(q-1)}{p}}  \biggl( \displaystyle\int_{B_{t}}(1+|D u_j|)^{\frac{np}{n-2 \lambda}}dx \biggr)^{\frac{(1-\delta)(n-2 \lambda)(q-1)}{np}} \notag\\
  &+ \biggl( \displaystyle\int_{B_{R}}(1+|D u_j|)^{p}dx \biggr)^{\frac{\delta q}{p}} \biggl( \displaystyle\int_{B_{t'}}(1+|D u_j|)^{\frac{np}{n-2 \lambda}}dx \biggr)^{\frac{(1-\delta)(n-2 \lambda)q}{np}} \biggr\}.
\end{align}

We need now to take the $L^\sigma$ norm with the measure $\frac{d h}{|h|^n}$ restricted to the ball $B(0,R/4)$ on the $h$-space of the $L^2$ norm of the difference quotient of order $\lambda$ of the function $V_p(Du_j)$. Since the functions $g_k$ are defined for $2^{-k}R/4 \leq |h| \leq 2^{-k+1}R/4$ we interpret the ball $B(0,R/4)$ as
$$ B(0,R/4)= \bigcup_{k=1}^{\infty} B(0,2^{-k+1}R/4)\setminus B(0,2^{-k}R/4)=: \bigcup_{k=1}^{\infty} E_k.$$
We obtain the following estimate

\begin{align}\label{Besovnorm}
 \ib0 & \biggl(\displaystyle\int_{B_s}  \dfrac{|\tau_hV_p(Du_j)|^2}{|h|^{2 \lambda}} dx \biggr)^{\frac{\sigma}{2}} \dfrac{d h}{|h|^n}\notag\\
  \leq 
& C   \biggl( \displaystyle\int_{B_{R}}(1+|D u_j|)^{p}dx \biggr)^{\frac{\delta (2q-p-2)\sigma}{2p}}\biggl( \displaystyle\int_{B_{t'}}(1+|D u_j|)^{\frac{np}{n-2 \lambda}}dx \biggr)^{\frac{(1-\delta )(n-2 \lambda)(2q-p-2)\sigma}{2np}} \notag	\\
& \cdot \ib0 \biggl( \displaystyle\int_{B_{R/2}}\dfrac{|\tau_h D \psi|^{2q-p}}{|h|^{\lambda(2q-p)}}dx \biggr)^{\frac{\sigma}{2q-p}} \dfrac{d h}{|h|^n}\notag\\
&+\dfrac{C}{(t-s)^\sigma}  \biggl( \displaystyle\int_{B_{R}}(1+|D u_j|)^{p}dx \biggr)^{\frac{\delta(2q-p)\sigma}{2p}}  \biggl( \displaystyle\int_{B_{t'}}(1+|D u_j|)^{\frac{np}{n-2 \lambda}}dx \biggr)^{\frac{(1-\delta)(n-2 \lambda)(2q-p)\sigma}{2np}} \notag\\
&\cdot \ib0 |h|^{(1-\lambda)\sigma} \dfrac{dh}{|h|^n} \notag\\
&+ C \biggl( \displaystyle\int_{B_{R}}(1+|D u_j|)^{p}dx \biggr)^{\frac{\delta(2q-p)\sigma}{2p}} \biggl( \displaystyle\int_{B_{t}}(1+|D u_j|)^{\frac{np}{n-2 \lambda}}dx \biggr)^{\frac{(1-\delta)(n-2 \lambda)(2q-p)\sigma}{2np}} \notag\\
& \cdot \displaystyle\sum_{k=1}^{\infty}\ibk |h|^{(\alpha-\lambda)\sigma} \biggl( \displaystyle\int_{B_{R/2}}(g_k(x+h)+g_k(x))^r dx \biggr)^{\frac{\sigma}{r}} \dfrac{dh}{|h|^n} \notag\\
&+ C \displaystyle\sum_{k=1}^{\infty}\ibk |h|^{(\alpha-\lambda)\frac{\sigma}{2}} \left( \displaystyle\int_{B_{R/2}} (g_k(x+h)+g_k(x))^r  dx \right)^\frac{\sigma}{2r} \left( \displaystyle\int_{B_{R/2}} \dfrac{|\tau_h D\psi|^{2q-p}}{|h|^{\lambda (2q-p)}} dx \right)^\frac{\sigma}{2(2q-p)} \dfrac{dh}{|h|^n} \notag\\
&\cdot \biggl( \displaystyle\int_{B_{R}}(1+|D u_j|)^{p}dx \biggr)^{\frac{\delta(q-1)\sigma}{2p}} \biggl( \displaystyle\int_{B_{t}}(1+|D u_j|)^{\frac{np}{n-2 \lambda}}dx \biggr)^{\frac{(1-\delta)(n-2 \lambda)(q-1)\sigma}{2np}}\notag\\
&+\frac{C}{(t-s)^{\sigma /2}} \displaystyle\sum_{k=1}^{\infty}\ibk |h|^{(\alpha+1-2 \lambda)\frac{\sigma}{2}} \left(\displaystyle\int_{B_{R/2}} (g_k(x+h)+g_k(x))^r dx \right)^\frac{\sigma}{2r} \dfrac{dh}{|h|^n}\notag\\
& \cdot \biggl\{ \biggl( \displaystyle\int_{B_{R}}(1+|D u_j|)^{p}dx \biggr)^{\frac{\delta(q-1)\sigma}{2p}}  \biggl( \displaystyle\int_{B_{t}}(1+|D u_j|)^{\frac{np}{n-2 \lambda}}dx \biggr)^{\frac{(1-\delta)(n-2 \lambda)(q-1)\sigma}{2np}}  \notag\\
  &+ \biggl( \displaystyle\int_{B_{R}}(1+|D u_j|)^{p}dx \biggr)^{\frac{\delta q\sigma}{2p}} \biggl( \displaystyle\int_{B_{t'}}(1+|D u_j|)^{\frac{np}{n-2 \lambda}}dx \biggr)^{\frac{(1-\delta)(n-2 \lambda)q\sigma}{2np}} \biggr\}.
\end{align}
Note that, since $\lambda \leq \gamma$, the integral 
$$J_1:= \ib0 \biggl( \displaystyle\int_{B_{R/2}}\dfrac{|\tau_h D \psi|^{2q-p}}{|h|^{\lambda(2q-p)}}dx \biggr)^{\frac{\sigma}{2q-p}} \dfrac{d h}{|h|^n}  $$
is controlled by the norm in the Besov space $B^{\gamma}_{2q-p,\sigma}$ on $B_{R/2}$ of the gradient of the obstacle which is finite by assumptions. The integral
$$ J_2 := \ib0 |h|^{(1-\lambda)\sigma} \dfrac{dh}{|h|^n}  $$ 
can be calculated in polar coordinates as follows
$$J_2 = C(n) \displaystyle\int_0^{R/4} \varrho^{(1-\lambda)\sigma-1} d \varrho = C(n,\alpha,\gamma,\sigma,R),$$
since $\lambda \in (0,1)$.
\\Now, we take care of the integral
$$J_3:= \displaystyle\sum_{k=1}^{\infty}\ibk |h|^{(\alpha-\lambda)\sigma} \biggl( \displaystyle\int_{B_{R/2}}(g_k(x+h)+g_k(x))^r dx \biggr)^{\frac{\sigma}{r}} \dfrac{dh}{|h|^n} .$$
Recalling that $|h| \leq 1$ and $\alpha \geq \lambda$, we have
$$J_3 \leq\displaystyle\sum_{k=1}^{\infty}\ibk  \biggl( \displaystyle\int_{B_{R/2}}(g_k(x+h)+g_k(x))^r dx \biggr)^{\frac{\sigma}{r}} \dfrac{dh}{|h|^n} .$$
We write the right hand sinde of the previous estimate in polar coordinates, so $h \in E_k$ if, and only if, $H= r \xi$ for some $2^{-k+1}R/4 \leq m < 2^{-k}R/4$ and some $\xi$ in the unit sphere $\mathbb{S}^{n-1}$ on $\R^n$. We denote by $d S(\xi)$ the surface measure on $\mathbb{S}^{n-1}$. We infer
\begin{align*}
J_3 \leq &  \displaystyle\sum_{k=1}^{\infty} \displaystyle\int_{m_{k-1}}^{m_k} \displaystyle\int_{\mathbb{S}^{n-1}}  \biggl( \displaystyle\int_{B_{R/2}}(g_k(x+h)+g_k(x))^r dx \biggr)^{\frac{\sigma}{r}} dS(\xi)\dfrac{dm}{m} \notag\\
= & \displaystyle\sum_{k=1}^{\infty} \displaystyle\int_{m_{k-1}}^{m_k} \displaystyle\int_{\mathbb{S}^{n-1}}  \Vert (\tau_{m \xi}g_k+g_k) \Vert_{L^{r}(B_{R/2})}^{\sigma} dS(\xi)\dfrac{dm}{m},
\end{align*}
where we set $m_k=2^{-k}\frac{R}{4}$.  We note that for each $\xi \in \mathbb{S}^{n-1}$ and $m_{k-1}\leq m \leq m_k$
\begin{align*}
\Vert (\tau_{m \xi}g_k+g_k) \Vert_{L^{r}(B_{R/2})}  \leq & \Vert g_k \Vert_{L^{r}(B_{R/2}-m_k \xi)} + \Vert g_k \Vert_{L^{r}(B_{R/2})} \notag\\
\leq & 2  \Vert g_k \Vert_{L^{r}(B_{R/2+R/4})},
\end{align*}
hence
$$J_3 \leq C(n) \Vert \{g_k \}_k \Vert_{l^\sigma(L^{r}(B_{R}))}^{\sigma},$$
which is finite by assumption (F6).
\\Recalling that $|h|\leq 1$, $\alpha \geq \lambda$ and using the Young's inequality with exponent $2$, we deduce the following estimate
\begin{align*}
& \displaystyle\sum_{k=1}^{\infty}\ibk |h|^{(\alpha-\lambda)\frac{\sigma}{2}} \left( \displaystyle\int_{B_{R/2}} (g_k(x+h)+g_k(x))^r  dx \right)^\frac{\sigma}{2r} \left( \displaystyle\int_{B_{R/2}} \dfrac{|\tau_h D\psi|^{2q-p}}{|h|^{\lambda (2q-p)}} dx \right)^\frac{\sigma}{2(2q-p)} \dfrac{dh}{|h|^n} \\
\leq & C \displaystyle\sum_{k=1}^{\infty}\ibk  \left( \displaystyle\int_{B_{R/2}} (g_k(x+h)+g_k(x))^r  dx \right)^\frac{\sigma}{r} \dfrac{dh}{|h|^n}+ C\ib0 \left( \displaystyle\int_{B_{R/2}} \dfrac{|\tau_h D\psi|^{2q-p}}{|h|^{\lambda (2q-p)}} dx \right)^\frac{\sigma}{2q-p} \dfrac{dh}{|h|^n} 
\end{align*} 
where the two integrals in the right hand side can be estimated as the integrals $J_1$ and $J_3$.
\\Similarly, we obtain
\begin{align*}
& \displaystyle\sum_{k=1}^{\infty}\ibk |h|^{(\alpha+1-2 \lambda)\frac{\sigma}{2}} \left(\displaystyle\int_{B_{R/2}} (g_k(x+h)+g_k(x))^r dx \right)^\frac{\sigma}{2r} \dfrac{dh}{|h|^n}\\
\leq &\ib0 |h|^{(\alpha+1-2 \lambda)\sigma} \dfrac{dh}{|h|^n} + \displaystyle\sum_{k=1}^{\infty}\ibk \left(\displaystyle\int_{B_{R/2}} (g_k(x+h)+g_k(x))^r dx \right)^\frac{\sigma}{r} \dfrac{dh}{|h|^n}.
\end{align*}
The latter term can be estimated as the integral $J_3$; the first integral can be calculated in polar coordinates as follows
$$J_2 = C(n) \displaystyle\int_0^{R/4} \varrho^{(\alpha+1-2 \lambda)\sigma-1} d \varrho = C(n,\alpha,\gamma,\sigma,R),$$
since $0 < \lambda \leq \alpha < 1$.

Estimate \eqref{Besovnorm} can be written in the following way
\begin{align}\label{Besovnorm1}
 \ib0 & \biggl(\displaystyle\int_{B_s}  \dfrac{|\tau_hV_p(Du_j)|^2}{|h|^{2 \lambda}} dx \biggr)^{\frac{\sigma}{2}} \dfrac{d h}{|h|^n}\notag\\
  \leq 
& \tilde{C}  \biggl( \displaystyle\int_{B_{R}}(1+|D u_j|)^{p}dx \biggr)^{\frac{\delta(2q-p-2)\sigma}{2p}}  \biggl( \displaystyle\int_{B_{t'}}(1+|D u_j|)^{\frac{np}{n-2 \lambda}}dx \biggr)^{\frac{(1-\delta )(n-2 \lambda)(2q-p-2)\sigma}{2np}}
\notag\\
&+\dfrac{\tilde{C}}{(t-s)^{\sigma }} \biggl( \displaystyle\int_{B_{R}}(1+|D u_j|)^{p}dx \biggr)^{\frac{\delta(2q-p)\sigma}{2p}} \biggl( \displaystyle\int_{B_{t'}}(1+|D u_j|)^{\frac{np}{n-2 \lambda}}dx \biggr)^{\frac{(1-\delta)(n-2 \lambda)(2q-p)\sigma}{2np}}   \notag\\
&+ \tilde{C} \biggl( \displaystyle\int_{B_{R}}(1+|D u_j|)^{p}dx \biggr)^{\frac{\delta(2q-p)\sigma}{2p}} \biggl( \displaystyle\int_{B_{t}}(1+|D u_j|)^{\frac{np}{n-2 \lambda}}dx \biggr)^{\frac{(1-\delta)(n-2 \lambda)(2q-p)\sigma}{2np}} \notag\\
&+ \tilde{C}  \biggl( \displaystyle\int_{B_{R}}(1+|D u_j|)^{p}dx \biggr)^{\frac{\delta(q-1)\sigma}{2p}} \biggl( \displaystyle\int_{B_{t}}(1+|D u_j|)^{\frac{np}{n-2 \lambda}}dx \biggr)^{\frac{(1-\delta)(n-2 \lambda)(q-1)\sigma}{2np}}\notag\\
&+\dfrac{\tilde{C}}{(t-s)^{\sigma /2}} \biggl( \displaystyle\int_{B_{R}}(1+|D u_j|)^{p}dx \biggr)^{\frac{\delta(q-1)\sigma}{2p}}  \biggl( \displaystyle\int_{B_{t}}(1+|D u_j|)^{\frac{np}{n-2 \lambda}}dx \biggr)^{\frac{(1-\delta)(n-2 \lambda)(q-1)\sigma}{2np}} \notag\\
  &+ \dfrac{\tilde{C}}{(t-s)^{\sigma /2}} \biggl( \displaystyle\int_{B_{R}}(1+|D u_j|)^{p}dx \biggr)^{\frac{\delta q\sigma}{2p}} \biggl( \displaystyle\int_{B_{t'}}(1+|D u_j|)^{\frac{np}{n-2 \lambda}}dx \biggr)^{\frac{(1-\delta)(n-2 \lambda)q\sigma}{2np}} \notag\\
  =& H_1+H_2+H_3+H_4+H_5+H_6,
\end{align}
for a constant $\tilde{C}:=	\tilde{C}(L,p,q,r,n,\sigma,\alpha,\gamma,R,\Vert D\psi \Vert_{B^{\gamma}_{2q-p,\sigma}(B_{R/2})}, \Vert \{g_k \}_k \Vert_{l^\sigma(L^{r}(B_{R}))})$.
\\We proceed estimating the various pieces arising up from \eqref{Besovnorm1}.

By assumption \eqref{gap2}, we have that
$$\dfrac{(1- \delta)(2q-p-2)}{p} <1 \quad \text{and} \quad \dfrac{(1- \delta)(2q-p)}{p} <1. $$
Thus, using the Young's inequality, we deduce the following estimate 
\begin{align}\label{H1}
H_1 +H_2+H_3\leq &\tilde{C}_{\theta} \biggl( \displaystyle\int_{B_{R}}(1+|D u_j|)^{p}dx \biggr)^{\sigma'}+ \tilde{C}_{\theta} \biggl( \displaystyle\int_{B_{R}}(1+|D u_j|)^{p}dx \biggr)^{\sigma''} \notag\\
& + \dfrac{\tilde{C}_\theta}{(t-s)^{\frac{\sigma p}{p-(1-\delta)(2q-p)}}} \biggl( \displaystyle\int_{B_{R}}(1+|D u_j|)^{p}dx \biggr)^{\sigma''} \notag\\
&+ 2\theta \biggl( \displaystyle\int_{B_{t'}}(1+|D u_j|)^{\frac{np}{n-2 \lambda}}dx \biggr)^{\frac{(n-2 \lambda)\sigma}{2n}} \notag\\
&+ \theta \biggl( \displaystyle\int_{B_{t}}(1+|D u_j|)^{\frac{np}{n-2 \lambda}}dx \biggr)^{\frac{(n-2 \lambda)\sigma}{2n}},
\end{align}
for $0< \theta < 1$, where we set $\sigma' = \frac{\delta (2q-p-2) \sigma}{2 [p-(1-\delta)(2q-p-2)]}$, $\sigma''= \frac{\delta (2q-p) \sigma}{2 [p-(1-\delta)(2q-p)]} $.
\\According to the second inequality of \eqref{ineq.1} with $\beta $ replaced by $\lambda$, the use of Young's inequality yields
\begin{align}\label{H2}
H_4 +H_5 \leq &\tilde{C}_{\theta} \biggl( \displaystyle\int_{B_{R}}(1+|D u_j|)^{p}dx \biggr)^{\sigma'''} + \dfrac{\tilde{C}_\theta}{(t-s)^{\frac{\sigma p}{2 [p-(1- \delta)(q-1)]} }} \biggl( \displaystyle\int_{B_{R}}(1+|D u_j|)^{p}dx \biggr)^{\sigma'''} \notag\\
& + 2\theta \biggl( \displaystyle\int_{B_{t}}(1+|D u_j|)^{\frac{np}{n-2 \lambda}}dx \biggr)^{\frac{(n-2 \lambda)\sigma}{2n}} ,
\end{align}
where we set $\sigma'' = \frac{p (q-1)\sigma}{2[p-(1-\delta)(q-1)]}$.
\\Similarly, recalling the third inequality of \eqref{ineq.1} with $\beta $ replaced by $\lambda$, we deduce that
\begin{align}\label{H3}
H_6 \leq & \dfrac{\tilde{C}_\theta}{(t-s)^{\frac{\sigma p}{2[p-(1-\delta)q]}}} \biggl( \displaystyle\int_{B_{R}}(1+|D u_j|)^{p}dx \biggr)^{\tilde{\sigma}}  + \theta \biggl( \displaystyle\int_{B_{t'}}(1+|D u_j|)^{\frac{np}{n-2 \lambda}}dx \biggr)^{\frac{(n-2 \lambda)\sigma}{2n}} ,
\end{align}
where we set $\tilde{\sigma} = \frac{p q \sigma}{2[p-(1-\delta)q]}$.

For a better readability we now define
\begin{align*}
A:= & \tilde{C}_{\theta} \biggl( \displaystyle\int_{B_{R}}(1+|D u_j|)^{p}dx \biggr)^{\sigma'}+ \tilde{C}_{\theta} \biggl( \displaystyle\int_{B_{R}}(1+|D u_j|)^{p}dx \biggr)^{\sigma''}, \notag\\
&+ \tilde{C}_{\theta} \biggl( \displaystyle\int_{B_{R}}(1+|D u_j|)^{p}dx \biggr)^{\sigma'''} , \notag\\
B_1:= & \tilde{C}_\theta \biggl( \displaystyle\int_{B_{R}}(1+|D u_j|)^{p}dx \biggr)^{\sigma''} ,\notag\\
B_2:= &\tilde{C}_\theta \biggl( \displaystyle\int_{B_{R}}(1+|D u_j|)^{p}dx \biggr)^{\sigma'''} ,\notag\\
B_3:= & \tilde{C}_\theta \biggl( \displaystyle\int_{B_{R}}(1+|D u_j|)^{p}dx \biggr)^{\tilde{\sigma}}, \notag\\
\pi_1:= & \dfrac{\sigma p}{p-(1-\delta)(2q-p)} ,\notag\\
\pi_2 := & \dfrac{\sigma p}{2 [p-(1- \delta)(q-1)]}, \notag\\
\pi_3 := & \dfrac{\sigma p}{2[p-(1-\delta)q]}, \notag\\
\end{align*}
so that, inserting estimates \eqref{H1}, \eqref{H2} and \eqref{H3} in \eqref{Besovnorm1}, we obtain
\begin{align}
\ib0 & \biggl(\displaystyle\int_{B_s}  \dfrac{|\tau_hV_p(Du_j)|^2}{|h|^{2 \lambda}} dx \biggr)^{\frac{\sigma}{2}} \dfrac{d h}{|h|^n}\notag\\
  \leq & 3\theta \biggl( \displaystyle\int_{B_{t}}(1+|D u_j|)^{\frac{np}{n-2 \lambda}}dx \biggr)^{\frac{(n-2 \lambda)\sigma}{2n}} \notag\\
  &+ 3\theta \biggl( \displaystyle\int_{B_{t'}}(1+|D u_j|)^{\frac{np}{n-2 \lambda}}dx \biggr)^{\frac{(n-2 \lambda)\sigma}{2n}} \notag\\
  &+ A+ \dfrac{B_1}{(t-s)^{\pi_1}} + \dfrac{B_2}{(t-s)^{\pi_2}}+ \dfrac{B_3}{(t-s)^{\pi_3 }} .
\end{align}
Lemma \ref{3.1} $(a)$ and inequality \eqref{Vp} imply
\begin{align}
\biggl( \displaystyle\int_{B_{s}} & |D u_j|^{\frac{np}{n-2 \lambda}}dx \biggr)^{\frac{(n-2 \lambda)\sigma}{2n}}  \notag\\
  \leq & 3\theta \biggl( \displaystyle\int_{B_{t}}(1+|D u_j|)^{\frac{np}{n-2 \lambda}}dx \biggr)^{\frac{(n-2 \lambda)\sigma}{2n}} \notag\\
  &+ 3\theta \biggl( \displaystyle\int_{B_{t'}}(1+|D u_j|)^{\frac{np}{n-2 \lambda}}dx \biggr)^{\frac{(n-2 \lambda)\sigma}{2n}} \notag\\
  &+  A+ \dfrac{B_1}{(t-s)^{\pi_1}} + \dfrac{B_2}{(t-s)^{\pi_2}}+ \dfrac{B_3}{(t-s)^{\pi_3 }} .
\end{align}
Arguing as in the proof of Theorem \ref{approximation}, we finally obtain
\begin{gather}
\biggl( \displaystyle\int_{B_{R/4}}  |Du_j|^{\frac{np}{n- 2 \lambda }} dx \biggr)^{\frac{(n- 2 \beta)\sigma}{2n}} \leq  \tilde{c}   \biggl\{  \ibR (1+ |Du_j|^p)dx+ \Vert D \psi \Vert_{B^{\gamma}_{2q-p, \sigma}(B_{R/2})} \biggl\}^{\kappa},
\end{gather}
which implies
\begin{gather}
\ib0  \biggl(\displaystyle\int_{B_{R/4}}  \dfrac{|\tau_hV_p(Du_j)|^2}{|h|^{2 \lambda}} dx \biggr)^{\frac{\sigma}{2}} \dfrac{d h}{|h|^n}dx \leq C  \biggl\{  \ibR (1+ |Du_j|^p)dx+ \Vert D \psi \Vert_{B^{\gamma}_{2q-p, \sigma}(B_{R/2})} \biggl\}^{\kappa},
\end{gather}
where $C := C(R,n,p,q,r, \sigma, \alpha,\gamma)$ and $\kappa:= \kappa (n,p,q,r,  \sigma, \alpha,\gamma)$. We observe that the constants $C$ and $\kappa$ are in particular independent of index $j$. Therefore, they are not an issue when passing to the limit in the approximating problem.

\medskip
\textbf{Acknowledgments}
The authors would like to thank Prof. Eleuteri and Prof. Passarelli di Napoli for suggesting the problem and for careful reading.

\end{document}